\documentclass[12pt]{article}
\usepackage{amsthm,amsmath,amsfonts,amssymb}
\usepackage{authblk}
\usepackage{fullpage}
\usepackage{mathtools,bm}
\usepackage{url}
\usepackage{bbm}
\usepackage{breqn}

\usepackage{mathrsfs} 
\usepackage{float}
\usepackage{todonotes}

\def\ignore#1{{}}

\def\bbbp{\mathbb{P}}

 \makeatletter
   \def\@fnsymbol#1{\ensuremath{\ifcase#1\or 1\or
   2\or 3\or \|\or **\or 4
   \or 5 \else\@ctrerr\fi}}
    \makeatother

\newtheorem{proposition}{Proposition}
\newtheorem{theorem}{Theorem}
\newtheorem{lemma}{Lemma} 
\newtheorem{corollary}{Corollary}

\newtheorem{remark}{Remark}
\newcommand{\p}{\mathbb{P}}
\newcommand{\e}{\mathbb{E}}
\newcommand{\N}{\mathbb{N}}
\newcommand{\Z}{\mathbb{Z}}
\newcommand{\A}{\mathcal{A}}

\newcommand{\mcL}{\mathcal{L}}

\providecommand{\keywords}[1]
{
  \small	
 \noindent\textbf{Keywords:} #1
}

\providecommand{\msc}[1]
{
  \small	
 \noindent\textbf{MSC:} #1
}

\begin{document}
\title{Markov chains arising from biased random derangements}

\author[a,c]{Poly H. da Silva\thanks{phd2120@columbia.edu}}
\affil[a]{{\footnotesize Columbia University, Department of Statistics, 1255 Amsterdam Avenue, New York, NY 10027, USA}}
\author[b,c]{Arash Jamshidpey\thanks{aj2963@columbia.edu}}
\affil[b]{{\footnotesize Columbia University, Department of Mathematics, 2990 Broadway, New York, NY 10027, USA}}
\author[a,c]{Simon Tavar\'e\thanks{st3193@columbia.edu}}
\affil[c]{{\footnotesize Columbia University, Irving Institute for Cancer Dynamics, Schermerhorn Hall, Suite 601, 1190 Amsterdam Avenue, New York, NY 10027, USA}}

\maketitle

\begin{abstract} 
We explore the cycle types of a class of biased random derangements, described as a random game played by some children labeled $1,\cdots,n$.  Children join the game one by one, in a random order, and randomly form some circles of size at least $2$, so that no child is left alone.  The game gives rise to the cyclic decomposition of a random derangement, inducing an exchangeable random partition.  The rate at which the circles are closed varies in time, and at each time $t$, depends on the number of individuals who have not played until t.  A $\{0,1\}$-valued Markov chain $ X^n$ records the cycle type of the corresponding random derangement in that any $1$ represents a hand-grasping event that closes a circle.  Using this,  we study the cycle counts and sizes of the random derangements and their asymptotic behavior. We approximate the total variation distance between the reversed chain of $X^n$ and its weak limit $X^\infty$, as $n\to\infty$. We establish conditional (and push-forward) relations between $X^n$ and a generalization of the Feller coupling, given that no $11$-pattern ($1$-cycle) appears in the latter. We extend these relations to $X^\infty$ and apply them to investigate some asymptotic behaviors of $X^n$.

\end{abstract}


\msc{60C05, 60J10, 60F05, 05A05, 65C40}




\section{Introduction}
This paper studies biased signed or unsigned random derangements as random permutations conditioned on having no fixed points. A simplified version of the random  derangement models studied in this paper may be described as the initial configuration of the following playground game played by $n$ children at camp. The children are labeled  $1,\dots,n$ and their right hands and left hands by $+1,\dots,+n$ and $-1,\dots,-n$, respectively. We denote by $|i|$ a child with hand $i\in [\pm n]:=\{\pm 1,\cdots,\pm n\}$. The goal is to form an ordered collection of circles of children holding hands, with some children looking in and others looking out of the circles. Of course no circle of size $1$ is allowed, which means a child is not allowed to grasp their other hand. The process is completed in $n$ steps. Each step corresponds to a hand grasping event. First, a label $\bar\sigma_n$ is chosen, uniformly at random, from $[+n]:=\{+1,\cdots,+n\}$. Child $|\bar\sigma_n|$ starts the game by initiating the first circle. This is represented by $(\bar\sigma_n,\dots$, where $\bar\sigma_n\in [+n]$  (i.e. it is positive) indicates that child $|\bar\sigma_n|$ looks in. While looking in, with their left hand child $|\bar\sigma_n|$ grasps a hand $\bar\sigma_{n-1}$ of another child chosen uniformly at random from $[\pm n]\setminus \{\pm|\bar\sigma_n|\}$. We record this as an incomplete circle $(\bar\sigma_n, \bar\sigma_{n-1},\dots$. Note that $\bar\sigma_{n-1}\in [+n]$  or it is equivalently a right hand, if and only if child $|\bar\sigma_{n-1}|$ looks in.  In the second step, child $|\bar\sigma_{n-1}|$, with their free hand  grasps a free hand chosen uniformly at random from the set of all remaining $2(n-2)+1$ free hands.  As a result, child $|\bar\sigma_{n-1}|$ chooses to grab $\bar\sigma_n$ (the right hand of $|\bar\sigma_n|$) with probability $1/(2(n-2)+1)$.  In this case,  the first circle is completed and $\bar\sigma_{n-2}$ is chosen uniformly at random from $[+n]\setminus\{\pm\bar\sigma_n,\pm\bar\sigma_{n-1}\}$. Child $|\bar\sigma_{n-2}|$ starts the second circle.  Again, we assume $|\bar\sigma_{n-2}|$ looks in. This is recorded as $(\bar\sigma_n,\bar\sigma_{n-1})(\bar\sigma_{n-2},...$. In the case that child $|\bar\sigma_{n-1}|$ chooses to not grab $\bar\sigma_n$, they grab a hand $\bar\sigma_{n-2}$ chosen uniformly at random from $[\pm n]\setminus \{\pm |\bar\sigma_n|,\pm |\bar\sigma_{n-1}|\}$.  We write $(\bar\sigma_n,\bar\sigma_{n-1},\bar\sigma_{n-2},\dots$ for the updated circle.  

In the same manner, right after the grasping event which results in the completion of a circle,  a new circle gets started by a child whose label is chosen uniformly at random from the set of all unused labels.   We always assume the child who starts a circle looks in, so we pick a positive label for the first child of each circle.  Note also that, in this way,  there is exactly one incomplete circle before every hand grasping event.  In each step,  the last child in the incomplete circle chooses a hand uniformly at random from the set of all remaining free hands which do not make any circle of size $1$, and grasps it with their free hand (left hand if they are the child who starts the circle). Assuming that $|\bar\sigma_2|$, the last-but-one child who joins the game, always grasps a randomly chosen hand of the last child $|\bar\sigma_1|$, so as not to leave the last child alone,  the process ends up with an ordered collection of circles of size at least $2$. We emphasize that the side at which a child looks is represented by signs $+/-$ such that $+i$ means the child $i\in [n]:=\{1,\dots,n\}$ looks in while $-i$ indicates the child $i$ looks out (cf. Figure~\ref{ex}).
\begin{figure}[!h]
\begin{center}
\includegraphics[scale=0.38]{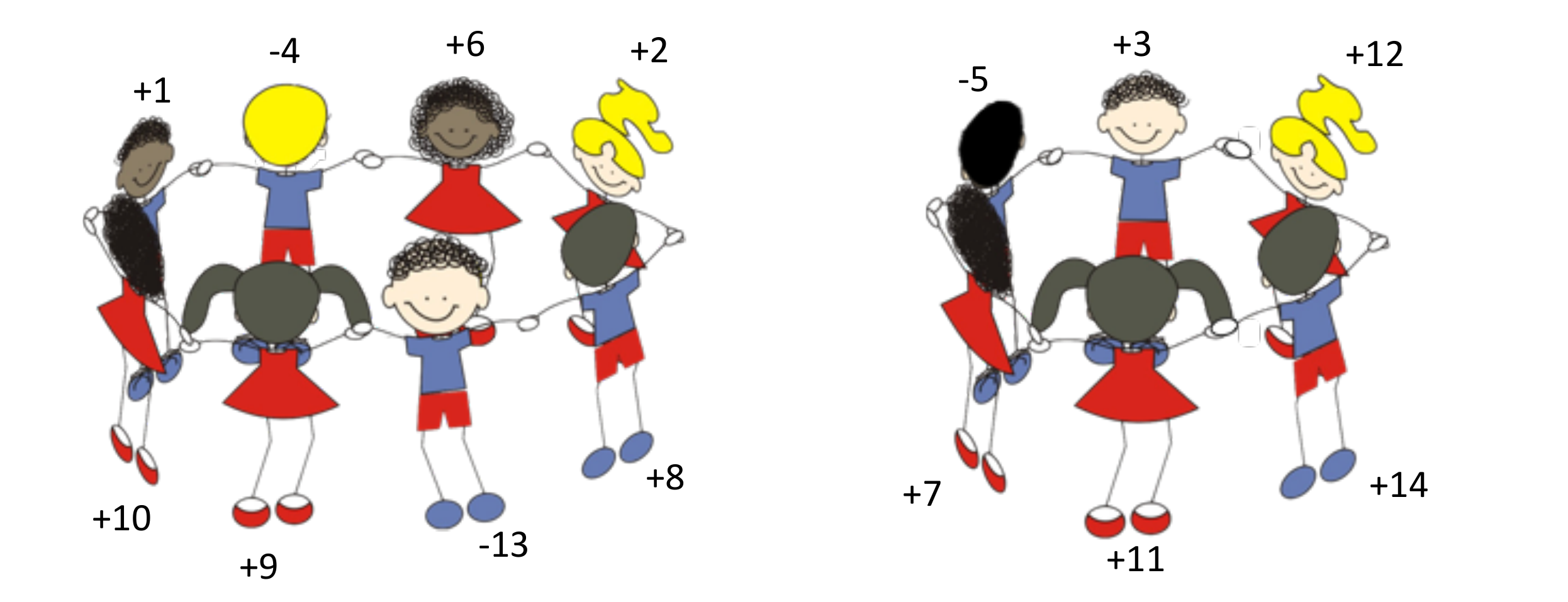} 
\end{center}
\caption{An example of a playground game with 14 children forming two circles with some children looking in and some looking out of the circle. The associated signed permutation is $\bar \sigma = (+1 -4 +6 +2 +8 -13 +9+10)(+3+12+14+11+7-5)$.  Modified figure from~\cite{fig}.}
\label{ex}
\end{figure} 

Note that when there are $n-j$ children left who are not playing yet and there are at least two children in the current incomplete circle, with their free hand the last child in the incomplete circle grasps any of the other $2(n-j)+1$ free hands (including the free hand of the first child in the incomplete circle) with the same probability, $0.5/(0.5+n-j)$.  We can generalize this model by changing the proportional weights of closing a circle versus grasping the right and left hands of children who are not playing yet.  More precisely, we assume that the weight of the right hand of the first child in the circle is given by $\theta$, while the right and the left hands of each of the other $n-j$ children with both hands free are given by $\kappa$ and $1-\kappa$, respectively. This means that the probability of closing the circle is $\theta/(\theta+n-j)$, and the probability that a new child joins the circle while they look in (or look out) is $\kappa (n-j)/(\theta+n-j)$ (or $(1-\kappa)(n-j)/(\theta+n-j)$, respectively).

The two-parameter weight structure allows us to formally define the playground game process in two steps. To see this,  let $p_1=0$, $p_2=1$, and  $p_i=(i-1)/(i-1+\theta)$ for $i\geq 3$. Let $\mathbf{p}=(p_i)_{i\geq 1}$ and $\mathbf{q}=(q_i)_{i\geq 1}$, where $q_i=1-p_i$. In the first step, we can model the sizes of the circles in the order of their formation via a $\{\pm 0,1\}$-valued Markov chain $\bar{X}^{n,\mathbf{p},\kappa}=(\bar{X}_{n+1-j}^n)_{j=0}^n$. More precisely, set $\bar{X}_{n+1}^{n}=1$ and for $j\leq n$,  at the $j^{th}$ step of the playground game, let $\bar{X}_{n+1-j}^{n}=1$ if the last child in the incomplete circle grasps the right hand of the first child in that circle and hence closes the circle, let $\bar{X}_{n+1-j}^{n}=+0$ if he or she grasps the right hand of a child with two free hands,  and let $\bar{X}_{n+1-j}^{n}=-0$ if he or she grasps the left hand of a child with two free hands. Letting $\bar{\mathcal{Q}}_r^{n,\mathbf{p}}(u,v):=\bbbp(\bar X_r^{n,\mathbf{p}} = v \mid \bar X_{r+1}^{n,\mathbf{p}} = u)$ for $1\leq r\leq n$ and $u,v\in \{\pm 0,1\}$, this can be formalized as a $\{\pm 0,1\}$-valued Markov chain with the transition probability matrices
$$\bar P_{r}^{n,\mathbf{p}} := \left(
\begin{array}{c@{~~~}c@{~~~}c}
\bar{\mathcal{Q}}_r^{n,\mathbf{p}}(+0, +0) & \bar{\mathcal{Q}}_r^{n,\mathbf{p}}(+0, -0) & \bar{\mathcal{Q}}_r^{n,\mathbf{p}}(+0, 1) \\[2pt]
\bar{\mathcal{Q}}_r^{n,\mathbf{p}}(-0, +0) & \bar{\mathcal{Q}}_r^{n,\mathbf{p}}(-0,-0 ) & \bar{\mathcal{Q}}_r^{n,\mathbf{p}}(-0, 1) \\[2pt]
\bar{\mathcal{Q}}_r^{n,\mathbf{p}}(1, +0) & \bar{\mathcal{Q}}_r^{n,\mathbf{p}}(1,-0) & \bar{\mathcal{Q}}_r^{n,\mathbf{p}}(1, 1) \\
\end{array}
\right)=\left(
\begin{array}{c@{~~~}c@{~~~}c}
\kappa p_r & (1-\kappa)p_r &q_r\\
\kappa p_r & (1-\kappa)p_r &q_r\\
\kappa &1-\kappa & 0\\
\end{array}
\right),$$
for $2\leq r\leq n-1$ and
$$
\bar P_n^{n,\mathbf{p}}= \left(
\begin{array}{c@{~~~}c@{~~~}c}
\kappa &1-\kappa & 0\\
\kappa &1-\kappa & 0\\
\kappa &1-\kappa & 0\\
\end{array}
\right),  \ \  \bar P_1^{n,\mathbf{p}}=\left(
\begin{array}{c@{~~~}c@{~~~}c}
0 & 0 & 1\\
0 & 0 & 1\\
0 & 0 & 1\\
\end{array}
\right).
$$
Having defined $\bar{X}^{n,\mathbf{p},\kappa}$,  in the second step, we use an auxiliary randomization to generate a random signed permutation representing the playground game process.  More explicitly,  choose $\bar\sigma_n$ unifromly at random from $[+n]$. Suppose we have $\bar{\sigma}_n,\cdots,\bar{\sigma}_{i+1}$ for $i<n$. Then pick $\bar\sigma_i$ unifromly at random from $\{+1,\cdots, +n\}\setminus \{\pm \bar{\sigma}_{i+1},\cdots, \pm\bar{\sigma}_n\}$ if $\bar{X}_{i+1}=1$ or $\bar{X}_{i+1}=+0$, and pick $\bar\sigma_i$ uniformly at random from $ \{- 1,\cdots,- n\}\setminus \{\pm \bar{\sigma}_{i+1},\cdots, \pm\bar{\sigma}_n\}$ if $\bar{X}_{i+1}=-0$. Denote by $K_n=\sum_{j=1}^n \bar{X}_j$ the number of circles in the $\bar{X}^{n,\mathbf{p},\kappa}$ process and by $A_j^{(n)}$, $j\leq K_n$,  the size of the $j^{th}$ circle in the $\bar{X}^{n,\mathbf{p},\kappa}$ process, in order of formation of circles. Letting $\tau_j=n+1-\sum_{i=1}^j A_j^{(n)}$, we can represent the random playground game with  parameters $\theta$ and $\kappa$ by
\[
(\bar\sigma_n,\cdots,\bar\sigma_{\tau_1})(\bar\sigma_{\tau_1+1},\cdots,\bar\sigma_{\tau_2})\cdots (\bar\sigma_{\tau_{k-1}+1},\cdots,\bar\sigma_{\tau_k}), \ \ k=K_n.
\]
Note that the circles are placed in order,  and $(x_1,\cdots,x_r)$ means that $|x_i|$ holds a hand of $|x_{i+1}|$, for $i=1,\cdots,r-1$ and $|x_r|$ holds the right hand of $|x_1|$. Furthermore, child $|x_i|$ looks in if and only if $x_i>0$. Forgetting the signs and the order among the circles in this representation, one obtains the cycle decomposition of a random permutation on $[n]$. It is clear from the definition that the random partition induced from the cycles of this random permutation is exchangeable. 
On the other hand, keeping the signs gives rise to a specific random signed permutation on $\{\pm 1,\cdots,,\pm n\}$, which will not be discussed further in this paper.

The case $p_i=(i-1)/(\theta+i-1)$ for $\theta>0$ is very special and can be extended to any $p_i\in (0,1)$ for $i\geq 3$. Henceforth, for $\mathbf{p}=(p_i)_{i=1}^\infty$ in the above construction, we assume that  $p_1=0$, $p_2=1$, and $p_i\in (0,1)$, for $i\geq 3$. 

The random signed permutations discussed above may be interpreted in various ways. For instance, one can consider $n$ strands, labeled $1,\cdots,n$. Each strand $i$ may represent a gene or a marker with two ends denoted by $-i$ and $+i$, where the signs indicate the genes' polarity or orientation (also called strandedness in the biological literature). As a result of the process discussed above, one can obtain a random genome with some circular chromosomes. See~\cite{bioinformatics} for further applications of signed permutations in Mathematical Genomics.  As another example, tying randomly the ends of $n$ cooked spaghetti strands such that the ends of any strand are not tied together (i.e. there is no circle of size $1$), one can construct a random spaghetti loop as discussed in~\cite{tavare2021}.  
In this paper, we only follow the playground game perspective.  Of particular interest is the probability distribution of cycle numbers and sizes, their asymptotic behavior, and their relationship to the so-called generalized Feller coupling conditional on having no cycle of size $1$.

\section{Description of results}
This section briefly discusses the main results of this paper. In particular, in Section~\ref{sec:unsigned}, we explain some results, mostly related to the number and type of the cycles in $\bar{X}^{n,\mathbf{p},\kappa}$, for which the sign information is not needed, hence can be ignored. This can be done via a projected $\{0,1\}$-valued Markov chain which is simply obtained from $\bar{X}^{n,\mathbf{p},\kappa}$ by dropping the signs from $-0$ and $+0$. In Section~\ref{sec:sign}, we explain how more complex quantities for $\bar{X}^{n,\mathbf{p},\kappa}$, specifically those for which the signs are involved, can be derived from the corresponding quantities for the projected chain. The outline of this paper will be given at the end of this section (see Section~\ref{sec:outline}).

\subsection{When signs do not matter}\label{sec:unsigned}
As already mentioned, we can ignore and drop the signs in studying those properties of $\bar{X}^{n,\mathbf{p},\kappa}$ for which signs are not involved. The major examples are the total number and the type of the cycles. More formally, by projecting $\{\pm 0,1\}$ onto $\{0,1\}$ such that $\pm 0\mapsto 0$ and $1\mapsto 1$, we derive a $\{0,1\}$-valued inhomogeneous Markov chain $X^{n,\mathbf{p}}:= (X_{n+1}=1, X_n,...,X_1)$ from $\bar{X}^{n,\mathbf{p},\kappa}$, whose transition matrices are given by 
\begin{equation*}\label{MCX}
P_{r}^{n,\mathbf{p}} :=\left(
\begin{array}{c@{~~~}c}
\bbbp(X_r^{n,\mathbf{p}} = 0 \mid X_{r+1} ^{n,\mathbf{p}}= 0) & \bbbp(X_r ^{n,\mathbf{p}}= 1 \mid X_{r+1}^{n,\mathbf{p}} = 0) \\[4pt]
\bbbp(X_r ^{n,\mathbf{p}}= 0 \mid X_{r+1}^{n,\mathbf{p}}= 1) & \bbbp(X_r ^{n,\mathbf{p}}= 1 \mid X_{r+1} ^{n,\mathbf{p}}= 1)\\
\end{array}
\right)= \left(
\begin{array}{c@{~~~}c}
p_r & q_r \\
1 & 0
\end{array}
\right),
\end{equation*}
for $r = n-1,\ldots,2$,  and
\begin{equation*}\label{MCX12}
P_{n}^{n,\mathbf{p}}= \left(
\begin{array}{c@{~~~}c}
1 & 0\\
1 & 0
\end{array}
\right),
\quad
P_{1}^{n,\mathbf{p}} = \left(
\begin{array}{c@{~~~}c}
0 & 1\\
0 & 1
\end{array}
\right),
\end{equation*}
where $p_1=0$, $p_2=1$, and $p_i\in (0,1)$, for $i\geq 3$. Note that under the above projection, the law of $X^{n,\mathbf{p}}$ can be obtained as the push-forward measure of the law of $\bar{X}^{n,\mathbf{p},\kappa}$.  For $\theta>0$, we denote by $\eta=\eta^{n,\theta}$ the special case of $X^{n,\mathbf{p}}$ with $p_r=(r-1)/(r-1+\theta)$. We also denote by $P_r^{n,\eta}=P_r^{n,\eta^\theta}, \ r=1,\cdots,n$, the transition probability matrices of $\eta$. In what follows, we refer to $X^{n,\mathbf{p}}$ and $\eta^{n,\theta}$ as the \emph{Random Derangement} and \emph{Playground Game} (PG) processes, respectively. 

In Section~\ref{sec:mc}, after some useful results such as the marginal distribution of $X^n$ and $\eta^n$ (Section~\ref{sec:marginalX}), we find the expected number of cycles and $j$-cycles and their asymptotics for $X^n$ and $\eta^n$; see Theorems~\ref{thm:meanX} and~\ref{cor:ECk}, Proposition~\ref{prop:limECeta}, and Lemmas~\ref{kn} and~\ref{Thm:ECk}.

In Section~\ref{sec:conditional}, we represent the law of the derangement Markov chain $X^{n,\mathbf{p}}$ as the law of the cycle type of a biased random permutation conditioned on not having any fixed points. In fact, we will see that the cycle type of this biased random permutation may be generated as the \textit{spacings} (each $j$-scpacing represents a $j$-cycle) between consecutive $1$s in a sequence of independent $\{0,1\}$-valued random variables called the \textit{Generalized Feller Coupling}. 

To better describe this, note that the dependency between the pairs of $\eta^{n,\theta}_{r+1}$ and $\eta^{n,\theta}_r$ (or more generally between $X^{n,\mathbf{p}}_{r+1}$ and $X^{n,\mathbf{p}}_r$) is a result of the discrepancy between the rows of $P_r^{n,\eta}$ (or $P_r^{n,\mathbf{p}}$). While the second row in $P_r^{n,\eta}$ ensures the occurance of no consecutive $1$s, the first row distributes the same chances to $0$ and $1$ as does the $r^{th}$ component of the classic \emph{Feller coupling}, introduced in~\cite{abt92}. So by replacing the lower row of $P_r^{n,\eta}$ by $(\frac{r-1}{r-1+\theta},\frac{\theta}{r-1+\theta})$, one can obtain the Feller coupling with parameter $\theta$ that is the sequence of independent Bernoulli random variables $\tilde{\xi}^\theta:=(\tilde{\xi}_i^\theta)_{i=1}^\infty$ with 
$
\p_\theta(\tilde{\xi}_i^\theta = 1)=1-\p_\theta(\tilde{\xi}_i^\theta = 0)=\frac{\theta}{i-1+\theta}
$.  As before, we drop $\theta$ if there is no risk of confusion. Let $C_j^{\tilde \xi}(n)$ be the number of $j$-spacings between $1$s in $1,\tilde{\xi}_n,\cdots, \tilde{\xi}_1$, i.e., the number of sub-patterns $1 \, 0^{j-1} \, 1$ in it.  The distribution of the cycle counts $C^{\tilde{\xi}}(n):=(C_1^{\tilde \xi}(n),\cdots,C_n^{\tilde \xi}(n))$ is given by the well-known Ewens Sampling Formula (ESF)~\cite{wje72}
\begin{equation*}
\begin{split}
ESF_n(\theta) (c_1,\cdots,c_n):=& \p_\theta(C_1^{\tilde \xi}(n)=c_1,\cdots,C_n^{\tilde \xi}(n)=c_n)\\
=& \mathbbm{1}\left(\sum\limits_{j=1}^n jc_j=n\right) \frac{n!}{\theta_{(n)}}\sum\limits_{j=1}^n \left(\frac{\theta}{j}\right)^{c_j} \frac{1}{c_j!}.
\end{split}
\end{equation*}
It is tempting to think that the law of $\eta^{n,\theta}$ is obtained from the $ESF_n(\theta)$ conditioned on $C_1^{\tilde \xi}(n)=0$, but this is not true.  Letting $\lambda_1(\theta) = 0$ and  
\begin{equation*}\label{thetahats}
 \lambda_n(\theta) := \bbbp_\theta(C_1^{\tilde{\xi}}(n) = 0) = \frac{n!}{\Gamma(n+\theta)} \sum_{j=0}^n (-1)^j \frac{\theta^j}{j!}\, \frac{\Gamma(n+\theta-j)}{(n-j)!}, n = 2, 3, \ldots
\end{equation*}
the authors showed in~\cite{sjt2021} that for a Markov chain $\tilde{\eta}=\tilde{\eta}^{n,\theta}=(\tilde{\eta}^{(n)}_{n+1}=1,\cdots,\tilde{\eta}^{(n)}_1)$ with the transition probabilities
$$
P_r^{\tilde{\eta}}=\left(
\begin{array}{c@{~~~}c}
\bbbp({\tilde\eta}^{(n)}_r = 0 \mid {\tilde\eta}^{(n)}_{r+1} = 0) & \bbbp({\tilde\eta}^{(n)}_r = 1 \mid {\tilde\eta}^{(n)}_{r+1} = 0) \\[4pt]
\bbbp({\tilde\eta}^{(n)}_r = 0 \mid {\tilde\eta}^{(n)}_{r+1} = 1) & \bbbp({\tilde\eta}^{(n)}_r = 1 \mid {\tilde\eta}^{(n)}_{r+1} = 1)
\end{array}
\right)
$$
defined by
\begin{equation*}\label{prdef}
P_r^{\tilde{\eta}}=\left(
\begin{array}{c@{~~~}c}
\displaystyle\frac{(\theta+r-1)\lambda_r(\theta)}{(\theta+r-1)\lambda_r(\theta) + \theta \lambda_{r-1}(\theta)} &  \displaystyle\frac{\theta\lambda_{r-1}(\theta)}{(\theta+r-1)\lambda_r(\theta) + \theta \lambda_{r-1}(\theta)} \\
&\\
1 & 0
\end{array}
\right),
\end{equation*}
for $2 < r < n$, and
$$
{P}^{\tilde{\eta}}_{n} = {P}^{\tilde{\eta}}_{2} = \left(
\begin{array}{c@{~~~}c}
1 & 0\\
1 & 0
\end{array}
\right),
\quad
{P}^{\tilde{\eta}}_{1} = \left(
\begin{array}{c@{~~~}c}
0 & 1\\
0 & 1
\end{array}
\right),
$$
we have the conditional relation 
\begin{equation*}\label{derlaw}
\mathcal{L}(C^{\tilde{\eta}}_2(n),\ldots,C^{\tilde{\eta}}_n(n)) = \mathcal{L}(C_2^{\tilde{\xi}}(n),\ldots,C_n^{\tilde{\xi}}(n) \mid C_1^{\tilde{\xi}}(n) = 0),
\end{equation*}
where $C^{\tilde{\eta}}_j(n)$ is the number of $j$-spacings between consecutive $1$s in $\tilde{\eta}$. 

A similar question can be asked about $\eta$ and more generally $X^{n,p}$: Is there a sequence of independent random variables $Y=(Y_i)_{i=1}^\infty$ for which $\p_\theta(Y=1)=1-\p_\theta(Y=0)=\theta/(i-1+\theta)$ for $i\in \N$, and for any $n\in \N$, the laws of $X^{n,\mathbf p}$ and $Y$ satisfy
\[
\mathcal{L}(X_n^{n,\mathbf{p}},\cdots,X_1^{n,\mathbf{p}})=\mathcal{L}(Y_n,\cdots,Y_1\mid C_1^Y(n)=0),
\] 
where $C_j^Y(n)$ is the number of $j$-spacings between $1$s (i.e. $1\ 0^{j-1}\ 1$ patterns) in the sequence $1, Y_n,\cdots,Y_1$? Theorem \ref{irvgeneral} shows that this is not necessarily possible for a fixed $\theta>0$. Instead, it shows that for a vector $\bm{\theta}:=(\theta_i)_{i\geq 1}$ with $\theta_i>0$ for $i\in \N$, defining $Y^{\bm\theta}=(Y_i^{\bm\theta})_{i=1}^\infty$ with $$\p_{\bm\theta}(Y_i=1)=1-\p_{\bm\theta}(Y_i=0)=\frac{\theta_i}{i-1+\theta_i}$$ and letting $C_j^{Y,\bm\theta}(n)$ count the number of $j$-spacings in $1, Y_n^{\bm\theta},\cdots, Y_1^{\bm\theta}$, we have
\begin{equation}\label{condlaw}
\mathcal{L}(X_n^{n,\mathbf{p}},\cdots,X_1^{n,\mathbf{p}})=\mathcal{L}(Y_n^{\bm\theta},\cdots,Y_1^{\bm\theta}\mid C_1^{Y,\bm\theta}(n)=0), 
\end{equation}
if and only if $\theta_i=(i-1) q_i/(p_ip_{i-1})$
for $i=3,...,n-1$. We call $Y^{\bm\theta}$ the generalized Feller coupling with vector parameter $\bm\theta$. The corresponding Feller coupling for $\eta^{n,\theta}$ is denoted by the sequence of independent $\{0,1\}$-valued random variables $\xi^\theta=(\xi_i^\theta)_{i=1}^\infty$ with $$\p_\theta(\xi_i^\theta=1)=1-\p_\theta(\xi_i^\theta=0)=\theta_i^*/(i-1+\theta_i^*)$$ where $\theta_i^*=\theta(1+\theta/(i-2))$, for $i\geq 4$, and $\theta^*_3=\theta$; cf. Corollary \ref{PGC}. Some basic properties of $Y^{\bm\theta}$ (e.g. its cycle count distribution) along with some of its applications for $X^{n,\textbf{p}}$ via the conditional relation are given in Sections~\ref{sec:conditional}, \ref{sec:cyclesRevisited} and \ref{sec:cycleCounts}. 

Recall that, to avoid $1$-cycles, we always assume $X^{n,\mathbf{p}}_n=0$, for any $n\in\N$, and this restriction is the reason that $\mathcal{L}(X^{n,\mathbf{p}})$ and $\mathcal{L}(X^{m,\mathbf{p}})$ cannot couple, for $n>m\geq 2$. In other words, $\mathcal{L}(X^{m,\mathbf{p}})$ cannot be obtained by projecting $\mathcal{L}(X^{n,\mathbf{p}})$ on the first $m$ components. Assuming $\sum_{j=1}^\infty p_j=\infty$, however ensures $\varphi_m(\mathbf{p}):=\lim_{n\to \infty}\p(X_m^n=1)$ exists and $0<\varphi_m(\mathbf{p})<1$ (Lemma~\ref{lemma-**}). Under this condition, we find an infinite Markov chain  $X^{\infty,\mathbf{p}}$ arising as the weak limit of $X^{n,\mathbf{p}}$, as $n\to\infty$. In particular, Theorem~\ref{TV-inf-ch} and Corollary~\ref{TV-converge} estimate the total variation distance between $X^{n,\mathbf{p}}$ and $X^{\infty,\mathbf{p}}$, and provides certain conditions under which  this distance converges to zero. In Section~\ref{sec:condRelationXinfinite}, we investigate conditions under which the conditional relation holds between  $X^{\infty,\mathbf{p}}$ and $Y^{\bm\theta}$,  namely $\mathcal{L}(X^{\infty,\mathbf{p}})=\mathcal{L}(Y^{\bm\theta}\mid \sum_{j=1}^\infty Y_j^{\bm\theta}Y_{j+1}^{\bm\theta}=0)$, where the condition on the right means there is no $11$ pattern (or equivalently, $1$-cycle) in the infinite chain $Y^{\bm\theta}$.

As discussed above, the conditional relation (\ref{condlaw}) holds for $\theta_i=(i-1)q_i/(p_ip_{i-1})$, not for $q_i=\theta_i/(i-1+\theta_i)$. For the latter, rather than a conditional relation, we have a push-forward relation under a specific $11$-erasing map. More precisely, let $\chi_n$ be a map that erases $11$ patterns, from left to right, from any given $\{0,1\}$-valued sequence $1,y_n,y_{n-1},\dots,y_2,y_1=1$. In Section~\ref{sec:coupling2} we will see that for $p_i=(i-1)/(i-1+\theta_i)$, $\mathcal{L}(X^{n,\mathbf{p}})=\chi_n\ast \mathcal{L}(Y_n^{\bm\theta},\cdots,Y_1^{\bm\theta})$, where the r.h.s. represents the push-forward of $\mathcal{L}(Y_n^{\bm\theta},\cdots,Y_1^{\bm\theta})$ under $\chi_n$. We extend this push-forward relation to $X^{\infty,\mathbf{p}}$ and apply it in Section~\ref{sec:clt} to prove a central limit theorem for the total number of cycles in $X^{n,\mathbf{p}}$, see Theorem~\ref{thm:clt} and Corollary~\ref{cor:clt}. Another application of this is given in Section~\ref{sec:orderedCycles}, where we show if $\theta_n\to\theta$ as $n\to\infty$, the normalized cycle-lengths of $X^{n,\mathbf{p}}$, in order of formation, converges weakly to $GEM(\theta)$; see Theorem~\ref{GEMlaw}.

The main purpose of this paper is to build a theory to relate $X^{n,\mathbf{p}}$, $X^{\infty,\mathbf{p}}$, and $Y^{\bm\theta}$ together in a very general set up where $\theta_i>0$ and $p_i\in (0,1)$, for $i\geq 3$. This is done using the conditional and push-forward relations which connects the random permutations and derangements.  These relations are applied to study the asymptotic behavior of the cycle counts of $X^{n,\mathbf{p}}$.  Although, most of the results hold for a broad range of choices for $\theta_i$ and $p_i$, when it comes to computational examples, we focus on calculating different quantities either for $\eta$ and $\xi$ where $\theta_i=\theta^*_i=\theta(1+\theta/(i-2))$, $i\geq 4$, or for $\tilde{\eta}$ and $\tilde{\xi}$ where $\theta_i=\theta$, $i\geq 3$.  It is worth mentioning that in addition to $\xi$ and $\tilde\xi$, another specific example of $Y^{\bm\theta}$, defined in this paper, has been studied by some authors.  More specifically,  \cite{ss2004}, \cite{holst2007} and \cite{holst2008} consider different versions of a sequence of independent Bernoulli random variables $\hat{Y}_1,\hat{Y}_2,\cdots$ for which $\p(\hat{Y}_i=1)=a/(b+i-1)$, for $0<a\leq b$ and $i\in \N$.  For $i\geq 2$, this is in fact equivalent to take $\theta_i=a(i-1)/(i-1+b-a)$ in our definition of $Y^{\bm\theta}$. Of course $\p(Y^{\bm\theta}_1=1)\neq\p( \hat{Y}_1=1)$ for $a<b$, but this is not a significant difference. \cite{hm2022} considers an extended version in which $\p(\hat{Y}_i=1)=a/(b+(i-1)^c)$, for $c>0$, that can be translated to $\theta_i=a(i-1)/(b-a+(i-1)^c)$, for $i\geq 2$. As already mentioned, although these are specific cases of $Y^{\bm\theta}$, we focus on a different set of problems in this paper.

\subsection{When signs should be tracked}\label{sec:sign}
In this section we briefly discuss the connections between the signed and unsigned models for more complicated problems in which signs are involved. We see how certain quantities of interest for the signed process can be easily obtained from the corresponding quantities for the unsigned model. To establish this, let $\mathscr X= (\mathscr X_1,\cdots,\mathscr X_n)$ be a general sequence of $n$, possibly correlated,  $\{\pm 0,1\}$-valued random variables, and for any $k=1,...,n$,  let $\mathscr C_k$ count the number of $k$-circles ($k$-spacings) in $\mathscr X$. Also, for $i=1,...,k$,  let $\mathscr{C}_{ki}(n)$ be the number of $k$-circles with exactly $i$ children looking in,  i.e. the number of those $k$ spacings with exactly $i-1$, $+0$ and $k-i$, $-0$ between the $1$s.  Suppose the random vectors $(\mathscr{C}_{1i}(n))_{i=1}^1,(\mathscr{C}_{2i}(n))_{i=1}^2,\cdots, (\mathscr{C}_{ni}(n))_{i=1}^n$ are conditionally independent given $\mathscr{C}_1,\cdots,\mathscr{C}_n$, and suppose for each $k$, $\mathscr{C}_{k1}(n),\dots, \mathscr{C}_{kk}(n)$ conditioned on $\mathscr{C}_k$, is multinomially distributed with weights $\omega_{ki}$, where $\sum_{i=1}^k \omega_{ki}=1$, that is the probability that there are exactly $i$ children looking in, in any given $k$-circle is $\omega_{ki}$.  For instance,  assuming the first child in each circle always looks in while any other child looks in or out with probability $\kappa$ and $1-\kappa$ resepctively, we get 
\begin{equation*}\label{binomial}
\omega_{ki}= {{k-1}\choose{i-1}} \kappa^{i-1}(1-\kappa)^{k-i}.
\end{equation*} 
Under the above assumptions, we can write
\begin{multline*}
\p(\mathscr{C}_{ki}=c_{ki}, \ k=1,...,n, \ i=1,...,k | \mathscr{C}_1=c_1,...,\mathscr{C}_n=c_n)=\\
\mathbbm{1}\{\sum\limits_{k=1}^n kc_k=n\} \prod\limits_{k=1}^n {{c_k}\choose{c_{k1},...,c_{kk}}} \prod\limits_{i=1}^k \left( \omega_{ki}\right)^{c_{ki}},
\end{multline*}
where $c_{ki}\in \Z_+$ and $c_k=\sum_{i=1}^k c_{ki}$ .
Note that there is at least one positive number in each circle, as the first child in each circle always looks in.  From the definition, 
$
\mathscr{C}_k=\sum_{i=1}^k \mathscr{C}_{ki}.
$
Then
\begin{multline*}
\p(\mathscr{C}_{ki}=c_{ki}, \ k=1,...,n, \ i=1,...,k) =\p( \mathscr{C}_1=c_1,...,\mathscr{C}_n=c_n) \times\\
\mathbbm{1}\left\lbrace\sum\limits_{k=1}^n kc_k=n\right\rbrace \prod\limits_{k=1}^n {{c_k}\choose{c_{k1},...,c_{kk}}} \prod\limits_{i=1}^k \left( \omega_{ki}\right)^{c_{ki}}.
\end{multline*}
In the special case that $\mathscr X$ is either $(Y_1^{\bm\theta},\cdots, Y_n^{\bm\theta})$ or $X^{n,\mathbf p}$,  the joint distribution of $\mathscr{C}_k$ is given by (\ref{jointDistY}) or (\ref{jointDistX}), respectively.  We can also find the distribution of $\mathscr{C}_{ki}$ once we know the distribution of $\mathscr{C}_{k}$. In fact,  
we can write
\begin{equation*}
\p(\mathscr{C}_{ki}=\ell)
=\mathbbm{1}\{k\ell\leq n\}\sum\limits_{m=\ell}^{\lfloor\frac{n}{k}\rfloor} {{m}\choose{\ell}}\omega_{ki}^\ell(1-\omega_{ki})^{m-\ell}\p(\mathscr{C}_k=m).
\end{equation*}

As another example, let $\mathscr{C}^*_j(n)$ be the number of circles in which exactly $j$ children look in.  Since $\mathscr{C}^*_j(n)=\sum_{k=j}^n\mathscr{C}_{kj}(n)$, we  have
\begin{eqnarray*}
\e \mathscr{C}^*_j(n) &=&\sum_{k\geq j}\sum_{r=0}^{\lfloor\frac{n}{k}\rfloor}\p(\mathscr{C}_k=r)\e(\mathscr{C}_{kj}\mid \mathscr{C}_k=r)\\
&=&\sum_{k\geq j} \omega_{kj}\sum_{r=0}^{\lfloor\frac{n}{k}\rfloor}r \p(\mathscr{C}_k=r)
= \sum_{k\geq j} \omega_{kj}\e(\mathscr{C}_k)
\end{eqnarray*}
and $\e(\mathscr{C}_k)$  is given in (\ref{meanX}) for $X^{n,\mathbf{p}}$ process and in (\ref{Eq-cor:ECK}) for $\eta$.

In addition, as $\mathscr{C}_{ki}$ and $\mathscr{C}_{k'j}$ are conditionally independent, given $\mathscr{C}_k$ and $\mathscr{C}_{k'}$, we have $\e({\rm Cov}(\mathscr{C}_{ki},\mathscr{C}_{k'j}\mid \mathscr{C}_1,\cdots,\mathscr{C}_n))=-\omega_{ki}\omega_{kj}\e(\mathscr{C}_k)\mathbbm{1}\{k=k'\}$. Hence
\begin{equation*}
\begin{split}
{\rm Cov}(\mathscr{C}^*_{i},\mathscr{C}^*_{j})
=& \sum_{k\geq i, k'\geq j}  \left\lbrace\e({\rm Cov}(\mathscr{C}_{ki},\mathscr{C}_{k'j}\mid \mathscr{C}_k,\mathscr{C}_{k'}))\right.\\
&~~~~~~~~~~~\left.+{\rm Cov}(\e(\mathscr{C}_{ki}\mid \mathscr{C}_k),\e(\mathscr{C}_{k'j}\mid \mathscr{C}_{k'}))\right\rbrace\\
=& \sum_{k\geq i, k'\geq j}  \omega_{ki}\omega_{k'j} {\rm Cov}(\mathscr{C}_k, \mathscr{C}_{k'})-\sum_{k\geq i \vee j}\omega_{ki}\omega_{kj}\e(\mathscr{C}_k).\\
\end{split}
\end{equation*}
Furthermore,  let $\mathscr{A}_i$ and $\mathscr{A}^*_i$ be the size of the $i$-th circle and the number of children looking in, in the  $i$-th circle.  Denote by $\mathscr{K}_n$ the total number of circles.  Then for $\sum_{i=1}^k a_i<n$,
\begin{eqnarray*}
\lefteqn{\p(\mathscr{A}^*_1=a_1,\dots, \mathscr{A}^*_k=a_k,\mathscr{K}_n>k)}\\
&&=\sum_{\substack{a_i\leq r_i \\ r_1+\dots+r_k<n}} \p(\mathscr{A}_1=r_1,\dots, \mathscr{A}_k=r_k,\mathscr{K}_n>k)\prod_{l=1}^k\p(\mathscr{A}^*_l=a_l\mid \mathscr A_l=r_l)\\
&&=\sum_{\substack{a_i\leq r_i \\ r_1+\dots+r_k<n}} \p(\mathscr{A}_1=r_1,\dots, \mathscr{A}_k=r_k,\mathscr{K}_n>k)\prod_{l=1}^k\omega_{r_la_l}
\end{eqnarray*}
where $ \p(\mathscr{A}_1=r_1,\dots, \mathscr{A}_k=r_k,\mathscr{K}_n>k)$ is given in (\ref{ordercyclesY}) for the $Y^{\bm\theta}$ process. 

Now, let $\Lambda_n$ be the total number of children looking in.  In the case that the probability of a child looking in is $\kappa$,  we have
\begin{equation*}
\p(\Lambda_n=r)
=\sum_{k=1}^r {{n-k}\choose{r-k}}\kappa^{r-k}(1-\kappa)^{n-r}\p(\mathscr{K}_n=k),
\end{equation*}
and $\p(\mathscr{K}_n=k)$ is given in (\ref{probKY}) and (\ref{probKX}) for $Y^{\bm\theta}$ and  $X^{n,\mathbf{p}}$ processes, respectively.  Also,
\begin{eqnarray*}
\e[\Lambda_n]&=&\e[\mathscr{K}_n]+\sum\limits_{i=1}^n\p(\mathscr{X}_i=+0)
=\e[\mathscr{K}_n]+n\kappa-\kappa\sum\limits_{i=1}^n \p(\mathscr{X}_i=1)\\
&=&n\kappa+(1-\kappa)\e[\mathscr{K}_n].
\end{eqnarray*}
Therefore,  knowing $\e[\mathscr{K}_n]$, we can obtain $\e[\Lambda_n]$.  For $X^{n,\mathbf{p}}$,  $\e[\mathscr{K}_n]$ is given in (\ref{meanKX}), and for $\eta$ in (\ref{EKeta}).

\subsection{Outline}\label{sec:outline}
The remainder of the paper is organized as follows. Section~\ref{sec:mc} explores some basic properties of  the Markov chain $X^{n,\textbf{p}}$ directly deduced from the transition probabilities. This includes the average number of cycles, the average number of $j$-cycles and their asymptotics.  Section~\ref{sec:coupling} provides necessary and sufficient conditions for the conditional law (\ref{condlaw}). The probability generating function of the total number of cycles and the joint distribution of the cycle counts are obtained using the conditional relation.  The last part of Section~\ref{sec:coupling} is devoted to finding the weak limit of $X^{n,\mathbf{p}}$ and its conditional relation with $Y^{\bm \theta}$, under certain conditions. Finally, Section~\ref{sec:coupling2} uses a specific coupling between $X^{n,\mathbf{p}}$ and $Y^{\bm\theta}$, for $p_i=(i-1)/(\theta_i+i-1)$, to derive a central limit theorem for the total number of cycles, and also to deduce the asymptotic behavior of the joint distribution of the normalized cycle lengths of the $X^{n,\textbf{p}}$ process, in order of their formation.

\section{Random derangements via Markov chains}\label{sec:mc}
In the following sections we study some properties of the Markov chain $X^{n,\mathbf{p}}=(1,  X_n,  \cdots, X_2, 1)$ derived directly from its transition matrices. We also explore the particular case of the playground game process $\eta$ via its corresponding chain.  As in the Feller coupling process, we denote by $C_j(n)$ the number of spacings of length $j$ between the $1$s in $X^{n,\mathbf{p}}$, that is equivalent to the number of cycles of size $j$. Also, we denote by $K_n$ the total number of cycles of $X^{n,\mathbf{p}}$.  We first compute the marginal distribution of $X_i$, and then use it to study the expected value of $K_n$ and its asymptotics. 

\subsection{Marginal distributions and transition probabilities}\label{sec:marginalX}
For $j\geq i\geq 1$ and $u,v\in \{0,1\}$, denote
$
\mathcal{Q}_{j,i}(u,v)=\p(X_i^{j-u,\mathbf{p}}=v).
$
From the Markov property, for any $n>j\geq i\geq 1$, and any $u,v\in \{0,1\}$ we have
\begin{equation}\label{X_Markov_relation}
\p(X_i^{n,\mathbf{p}}=v\mid X_j^{n,\mathbf{p}}=u)=\mathcal{Q}_{j,i}(u,v),
\end{equation}
and as $\p(X_{n+1}^{n,\mathbf{p}}=0)=\p(X_{n}^{n,\mathbf{p}}=1)=0$, we assume by convention that
(\ref{X_Markov_relation}) holds for any $n+1\geq j\geq i\geq 1$. On the other hand,  it is clear from the definition of $X^{n,\mathbf{p}}$ that 
$
\mathcal{Q}_{ji}(0,1)=\mathcal{Q}_{j+1,i}(1,1)
$, for $j\geq i$. 
Hence, to determine the values of $\mathcal{Q}_{ji}(u,v)$ for any $j>i$ and any $u,v\in \{0,1\}$, we need to find the values of $\mathcal{Q}_{j+1,i}(1,1)=\p(X_i^{j,\mathbf{p}}=1)$, for any $j\geq i$. 
The following lemma may be proved by induction on $i$ and using the fact that 
\[
\p(X_{i}^{n,\mathbf{p}}=1) = \sum\limits_{u\in \{0,1\}} \p(X_{i+1}^{n,\mathbf{p}}=u)\mathcal{Q}_{i+1,i}(u,1)
= q_{i}(1-\p(X_{i+1}^{n,\mathbf{p}}=1)).
\]
\begin{lemma}\label{ProbEta}
For $3\leq i< n$,
\begin{equation*}
\p(X_i^{n,\mathbf{p}}=1)=\sum\limits_{j=0}^{n-i-1}(-1)^j \prod\limits_{l=i}^{j+i}q_l.
\end{equation*}
\end{lemma}
Summarizing the above discussion,  for $j>i>2$ and $u,v\in \{0,1\}$, we get
\[
\begin{split}
\mathcal{Q}_{ji}(u,v)=& \p(X_i^{j-u,\mathbf{p}}=v)\\
= & v \left( \sum\limits_{r=0}^{j-u-1-i} (-1)^r \prod_{s=i}^{i+r} q_s\right)+(1-v)\left(1-\sum\limits_{r=0}^{j-u-1-i} (-1)^r \prod_{s=i}^{i+r} q_s\right)\\
= & (1-v)+(2v-1) \sum\limits_{r=0}^{j-u-1-i} (-1)^r \prod_{s=i}^{i+r} q_s,
\end{split}
\]
where, by convention, we assume $\sum_{r=0}^{-1}  a_r=0$. 
In particular,  for $i_0=n+1>i_1>\cdots>i_k>2$, we have
\[
\e_{n,\mathbf{p}}[ X_{i_1}\ldots X_{i_k}]=\prod_{l=0}^{k-1} \sum\limits_{r=0}^{i_l-i_{l+1}-2} (-1)^r \prod_{s=i_{l+1}}^{r+i_{l+1}} q_s.
\]

For $a,b,z\in \mathbb{C}$, denote by
\[
M(a,b,z)=\sum\limits_{j=0}^\infty \frac{a_{(j)} z^j}{b_{(j)} j!},
\]
the confluent hypergeometric function. For $Re \ b> Re \ a>0$, the integral representation of $M$ is given by
\begin{equation}\label{hyp-int}
M(a,b,z)=\frac{\Gamma(b)}{\Gamma(a)\Gamma(b-a)}\int_0^1 e^{zu} u^{a-1}(1-u)^{b-a-1} \ du.
\end{equation}
We have the following lemma.
\begin{lemma}\label{marginal-eta} For any $3\leq i< n$,
\[
\p_\theta(\eta_i=1)=\sum\limits_{j=1}^{n-i} (-1)^{j+1} \frac{\theta^j}{(\theta+i-1)_{(j)}}.
\]
Furthermore,  for any $i\geq 3$
\begin{equation*}
\lim\limits_{n\rightarrow\infty} \p_\theta(\eta_i=1)=\frac{\theta}{\theta+i-1} M(1,\theta+i,-\theta)=\theta\int_0^1e^{-\theta u} (1-u)^{\theta+i-2} du.
\end{equation*}
\end{lemma}
\begin{proof}
The first part is clear from Lemma \ref{ProbEta} by letting $q_l=\theta / (\theta+l-1)$.  For the second part, write
\[
\begin{split}
\lim\limits_{n\rightarrow \infty}\p(\eta_i =1)=-\sum\limits_{j=1}^\infty \frac{(-\theta)^j}{(\theta+i-1)_{(j)}} & =1-\sum\limits_{j=0}^\infty \frac{1_{(j)}(-\theta)^j}{(\theta+i-1)_{(j)}j!}\\
& = 1- M(1,\theta+i-1,-\theta). 
\end{split}
\]
Using (\ref{hyp-int}) and integrating by parts, the second part of the lemma follows.
\end{proof}

For $i_0=n+1>i_1>\cdots>i_k>2$,
\[
\e_\theta [\eta_{i_1}\ldots \eta_{i_k}]=\prod_{r=1}^k \sum\limits_{l=0}^{i_{r-1}-i_r-2} (-1)^l \frac{\theta^{l+1}}{(\theta+i_{r}-1)_{(l+1)}}.
\]

As a result, for $2<i<j<n-1$,
\begin{equation*}
\begin{split}
{\rm Cov}(\eta_j, \eta_i)=&  \frac{(-1)^{i+j+1} \Gamma(\theta+i-1)\Gamma(\theta+j-1)}{\theta^{i+j-2}}\\
& \times \left( \sum\limits_{l=j}^{n-1} (-1)^l \frac{\theta^l}{\Gamma(\theta+l)}\right) \left( \sum\limits_{l=j-1}^{n-1} (-1)^l \frac{\theta^l}{\Gamma(\theta+l)}\right).
\end{split}
\end{equation*}
\subsection{Number of cycles}
Let $H_n:=\sum_{i=1}^n 1/i$ and let $\gamma:= \lim\limits_{n \rightarrow \infty} H_n- \log n \simeq0.5772$  denote the Euler constant. As in the Feller Coupling case,  the following lemma shows  the linear relationship of the expected value of $K_n$ and $\log n$, under specific conditions.
 \begin{lemma}\label{kn}
For $n>3$,
\begin{equation}\label{meanKX}
\e_{\mathbf{p}}(K_n)=\sum\limits_{i=1}^{n-1}\sum\limits_{j=0}^{n-i-1}(-1)^j \prod\limits_{l=i}^{j+i}q_l.
\end{equation}

If, in addition, there exists a constant $\alpha\geq 0$ such that $n q_n\to \alpha$, as $n\to\infty$, and $|\psi_\alpha(\mathbf{p})|<\infty$ for
\begin{equation}\label{psifunction}
\psi_\alpha(\mathbf{p})=\psi(\mathbf{p}):= \sum_{i=1}^\infty \left(q_i-\frac{\alpha}{i}\right),
\end{equation}
then

 \begin{equation}\label{ineq}
\lim_{n\rightarrow \infty} \e_{\mathbf{p}}(K_n) - \alpha\log n =\alpha\gamma+\psi(\mathbf{p})+\sum\limits_{j=1}^\infty (-1)^j \bar a_j,
 \end{equation}
where $\bar a_j= \sum\limits_{i=1}^\infty \prod\limits_{l=i}^{j+i}q_l<\infty$, for $j\in \N$.
\end{lemma}

\begin{proof}
The first part is straightforward  from Lemma~\ref{ProbEta}. Now note that for any fixed $j\geq 1$, $\lim\limits_{i\rightarrow \infty}i_{(j+1)}\prod\limits_{l=i}^{j+i}q_l=\alpha^j\geq 0$. Therefore, from the limit comparison test,  as $\sum_{i=1}^\infty 1/i_{(j+1)}$ converges  so does $\bar a_j$.
We have
$$
\sum_{i=1}^{n-2}\sum_{j=1}^{n-i-1}(-1)^{j+1} \prod\limits_{l=i}^{j+i}q_l\leq \sum_{i=1}^{n-2}q_iq_{i+1}\leq \sum_{i=1}^\infty q_iq_{i+1}< \infty.
$$
Also one can easily see that the l.h.s. is a positive increasing function of $n$, hence the limit of the l.h.s., as $n\to\infty$, exists and is finite. To find this limit, note that $a_j(n):=\sum_{i=1}^{n-j-1}\prod_{l=i}^{j+i} q_l$ decreases as $j$ increases, since $q_i<1$ for $i\geq 3$. Hence, by interchanging the sums, for any $m\geq 1$,
\begin{equation}\label{lemma3-upperbound}
 \lim_{n\to \infty}\sum_{i=1}^{n-2}\sum_{j=1}^{n-i-1}(-1)^j\prod_{l=i}^{i+j}q_l= \lim_{n\to \infty}\sum_{j=1}^{n-2}(-1)^ja_j(n)
\leq \sum_{j=1}^{2m-1} (-1)^j \bar a_j+\bar{a}_{2m},  
\end{equation}
and similarly, the l.h.s. of (\ref{lemma3-upperbound}) is greater than or equal to
\begin{equation}\label{lemma3-lowerbound}
    \sum\limits_{j=1}^{2m+1} (-1)^j \bar a_j \geq  \sum\limits_{j=1}^{2m-1} (-1)^j \bar a_j  - \bar{a}_{2m+1}.
\end{equation}
Hence, to see that the l.h.s. of (\ref{lemma3-upperbound}) equals $\sum_{i=1}^\infty (-1)^j \bar{a}_j$, it suffices to show $\bar{a}_{2m}\to 0$, as $m\to\infty$. This is clear from
\[
\sum\limits_{i=1}^\infty \prod_{l=i}^{i+j} q_l\leq \sum\limits_{i=1}^\infty q_iq_{i+1}<\infty,  \ \ j=2,3,\cdots,
\]
and
\[
\lim\limits_{j\to \infty} \bar a_j=\lim\limits_{j\to\infty}\sum\limits_{i=1}^\infty \prod_{l=i}^{i+j} q_l=\sum\limits_{i=1}^\infty\lim\limits_{j\to\infty} \prod_{l=i}^{i+j} q_l=0.
\]
Now, write
\begin{eqnarray*}
\lim_{n\to\infty}\e_{\mathbf{p}}(K_n) - \alpha\log n &=&\lim_{n\to\infty} \left(\sum_{i=1}^{n-1}q_i-\alpha H_n\right)+\alpha\left(\lim_{n\to\infty}H_n-\log n \right)\\
&&-\lim_{n\to\infty}\sum_{i=1}^{n-2}\sum_{j=1}^{n-i-1}(-1)^{j+1}\prod_{l=i}^{i+j}q_l\\
&=& \psi(\mathbf{p})+\alpha\gamma+\sum\limits_{j=1}^\infty (-1)^j \bar a_j.
\end{eqnarray*}
\end{proof}
\begin{remark}
Note that (\ref{lemma3-upperbound}) and (\ref{lemma3-lowerbound}) give upper and lower bounds for the sum on the right of (\ref{ineq}). 
\end{remark}
Denote by $K_n^{\eta}$ the number of cycles of the playground game $\eta^n$. The following result is an application of Lemma~\ref{kn}. Before stating the next theorem, recall that for any $p,q\in \N$, the generalized hypergeometric function $_pF_q$ is defined by
\[
_pF_q(a_1,\cdots,a_p; b_1,\cdots,b_p; z)=\sum\limits_{j=0}^\infty \frac{(a_1)_{(j)}\cdots(a_p)_{(j)}z^j}{(b_1)_{(j)}\cdots(b_q)_{(j)}j!}.
\]
Note that $_1F_1(a;b;z)=M(a,b,z)$ the confluent hypergeometric function. From the Euler's integral transform
\begin{multline}\label{hyp-int-2}
    _{p+1}F_{q+1}(a_1,\cdots,a_p,a; b_1,\cdots,b_p,b; z)=\\
    \frac{\Gamma(b)}{\Gamma(a)\Gamma(b-a)}\int_0^1 x^{a-1}(1-x)^{b-a-1} \ _pF_q(a_1,\cdots,a_p; b_1,\cdots,b_p; zx) \ dx.
\end{multline}
\begin{theorem}\label{thm:meanX}
For $n>3$,
\begin{equation}\label{EKeta}
\e(K_n^{\eta})=1+\sum\limits_{i=3}^{n-1}\frac{\theta}{\theta+i-1}+\sum\limits_{j=2}^{n-3}\sum\limits_{i=3}^{n-j}\frac{(-1)^{j+1}\theta^j}{(\theta+i-1)_{(j)}},
\end{equation}
\begin{equation}\label{ineq-eta}
\lim_{n\rightarrow \infty} \e(K_n^{\eta}) - \theta \log n= 1-\theta H_{\theta+1} +\theta \gamma -\theta^2 \int_0^1\int_0^1 e^{-\theta xy}(1-x)^{\theta+1} dx \ dy,
 \end{equation}
 where \[H_y=\int_0^1 \frac{1-x^y}{1-x}dx.\]
\end{theorem}

\begin{proof}
Letting $q_i=\theta/(\theta+i-1)$, for $i\geq 3$, we have $\lim_{n\rightarrow \infty}nq_n=\theta$, and
\[
\begin{split}
    \psi_\theta(\mathbf{p})=&\lim_{n\to\infty} \left(-\theta H_{n-1}+1+\sum_{i=3}^{n-1}\frac{\theta}{\theta+i-1}\right)\\
=&1+\theta \lim\limits_{n\rightarrow\infty}\left(- H_{n-1}+H_{n-2+\theta}-H_{\theta+1}\right)=1-\theta H_{\theta+1},
\end{split}
\]
where the last equality follows from $\lim_{n\rightarrow \infty}(H_{n-2+\theta}-H_{n-1})=0$. Also,
$$
\bar a_j=\lim_{n\to\infty}a_j(n)=\sum_{i=3}^\infty\frac{\theta^{j+1}}{(\theta+i-1)_{(j+1)}}=\frac{\theta^{j+1}}{j(\theta+2)_{(j)}}.
$$
Now we have
$$\sum\limits_{j=1}^\infty (-1)^j \bar a_j =\frac{-\theta^2}{\theta+2}\sum\limits_{j=0}^\infty \frac{1_{(j)}1_{(j)}(-\theta)^j}{2_{(j)}(\theta+3)_{(j)}j!}=\frac{-\theta^2}{\theta+2} \ _2F_2(1,1; 2,\theta+3;-\theta),
$$
where from (\ref{hyp-int-2}), the l.h.s. reduces to
\[
\begin{split}
\frac{-\theta^2 \Gamma(\theta+3)}{(\theta+2)\Gamma(1)\Gamma(\theta+2)}\int_0^1 (1-x)^{\theta+1} M(1,2,-\theta x) \ dx,
\end{split}
\]
which, from (\ref{hyp-int}), in turn reduces to the term with the double integral on the right of (\ref{ineq-eta}). Applying Lemma~\ref{kn} completes the proof.
\end{proof}
Note that, as before, the double integral on the right of (\ref{ineq-eta}) is bounded by
 \[
\sum\limits_{j=2}^{2m+1}\frac{(-1)^{j+1}\theta^j}{(j-1)(\theta+2)_{(j-1)}},
 \]
for any $m\in \N$. As an example, for $\theta=1/2$ and $m=3$, we obtain
$\lim_{n\rightarrow \infty} \e(K_n^{\eta}) - \frac{1}{2} \log n\approx 0.555069$,
 with the error $\leq 1.23333\times 10^{-7}$.


\subsection{Cycles of size $j$}
The following lemma gives the expected value of $C_j(n)$ and its asymptotics.
\begin{lemma}\label{Thm:ECk}
For $j>1$,
\begin{equation}\label{meanX}
\e_{\mathbf{p}}(C_{j}{(n)})= q_{n-j+1}\prod\limits_{l=n-j+2}^{n-1}p_l+\sum\limits_{i=j+1}^{n-1}\sum\limits_{k=0}^{n-i-1}(-1)^kq_{i-j}\prod\limits_{l=i-j+1}^{i-2}p_l\prod\limits_{l=i}^{k+i}q_l.
\end{equation}
Furthermore, if there exists a constant $\alpha \geq 0$ such that $\lim\limits_{n\rightarrow \infty}n q_n=\alpha$, then for any $m\geq 1$
\begin{equation*}
\lim_{n\longrightarrow \infty} \e_{\mathbf{p}}(C_j{(n)})=\sum\limits_{k=0}^{2m-1}(-1)^{k} \bar b_k(j)+\varepsilon(m,\mathbf{p},j),
\end{equation*}
with $\varepsilon(m,\mathbf{p},j)\leq  \bar b_{2m}(j)$, where 
\begin{equation*}
\bar b_{k}(j)= \lim_{n \to \infty}b_k(n,j)=\lim_{n \to \infty} \sum\limits_{i=j+1}^{n-k-1}q_{i-j}\prod\limits_{l=i-j+1}^{i-2}p_l\prod\limits_{l=i}^{k+i}q_l.
\end{equation*}
\end{lemma}

\begin{proof}
Write
\begin{multline*}\label{qki}
\mathcal R_{j,i}(n):=\p_{\mathbf{p}}\left(X_i^{n}X_{i-j}^{n}\prod_{l=1}^{j-1}\left(1-X_{i-l}^{n}\right)=1\right)=\p(X_i^{n}=1) q_{i-j}\prod\limits_{l=i-j+1}^{i-2}p_l\\
= \sum\limits_{k=0}^{n-i-1}(-1)^k q_{i-j}\prod\limits_{l=i-j+1}^{i-2}p_l\prod\limits_{l=i}^{k+i}q_l,
\end{multline*}
for $j+1\leq i\leq n-1$, and
\begin{equation*}\label{qkk}
\mathcal R_{j,n+1}(n)= q_{n-j+1}\prod\limits_{l=n-j+2}^{n-1}p_l.
\end{equation*}

The first part of the theorem follows from
$$
\e_{\mathbf{p}}(C_{j}{(n)})=\sum\limits_{i=j+1}^{n+1} \mathcal{R}_{j,i}(n).
$$

For the second part, note that for any $j\geq 2$, $\mathcal{R}_{j,n+1}(n)\to 0$, as $n\to\infty$. As $q_1=1$
\begin{eqnarray*}
\sum_{i=j+1}^{n-1}\sum_{k=0}^{n-i-1}(-1)^k q_{i-j}\prod\limits_{l=i-j+1}^{i-2}p_l\prod\limits_{l=i}^{k+i}q_l &\leq &\sum_{i=j+1}^{n-1} q_iq_{i-j}\prod\limits_{l=i-j+1}^{i-2}p_l\\
&\leq & \sum_{i=j+1}^\infty q_iq_{i-j}<\infty,
\end{eqnarray*}
where the l.h.s. is positive. In fact, for any fixed $j>1$, there exists $N_j$ large enough that the l.h.s. of the last inequalities is an increasing function of $n$, for $n\geq N_j$. To see this, temporarily denote by $h_j(n)$  the l.h.s. of the last inequality. Then
$$
h_j(n+1)-h_j(n)=\sum_{i=j+1}^n(-1)^{n-i}q_{i-j}\prod_{l=i}^nq_l\prod_{l=i-j+1}^{i-2}p_l= q_nq_{n-j}\prod_{l=n-j+1}^{n-2}p_l+o(n^{-2}),
$$
for large enough $n$. Thus,  $\lim_{n\to\infty}h_j(n)$ and therefore $\lim_{n\to\infty}\e_{\mathbf{p}}(C_j(n))$ exist and are finite. On the other hand, as in the proof of Lemma~\ref{kn}, from the limit comparison test $\bar b_k(j)=\lim_{n\to\infty}b_k(n,j)$ exists and is finite. Now, as $b_k(n,j)$, for fixed $n$, is a decreasing sequence as $k$ increases, for any $m\geq 1$,
\begin{eqnarray*}
\lim_{n\to\infty}\e_{\mathbf{p}}(C_j(n))&=&\lim_{n\to\infty} h_j(n)=\lim_{n\to\infty}\sum_{k=0}^{n-j-2}(-1)^kb_k(n,j)\\
&=& \sum_{k=0}^{2m-1}(-1)^k \bar b_k(j)+\varepsilon(m,\mathbf{p},j),
\end{eqnarray*}
where $\varepsilon(m,\mathbf{p},j)\leq \bar b_{2m}(j)$.
\end{proof}

Let $C_{j}^{\eta}(n)$ be the number of the cycles of size $j$ in the playground game $\eta^n$.  Applying Lemma~\ref{Thm:ECk}  to the $\eta$ process,  we obtain the following theorem. By convention, we let $\theta_{(-1)}=\theta_{(0)}=1$ and $\sum_{i=m}^n a_i=0$, for $m>n$.

\begin{theorem}\label{cor:ECk}
For $2\leq j\leq n-2$,
\begin{multline}\label{Eq-cor:ECK}
\e_\theta(C^{\eta}_{j}(n))= \frac{\theta (n-j+1)_{(j-2)}}{(\theta+n-j)_{(j-1)}}+\sum\limits_{k=1}^{n-j-1}\frac{(-1)^{k+1}\theta^k(j-2)!}{(\theta+2)_{(j-3)}(\theta+j)_{(k)}}\\
+\sum\limits_{i=j+3}^{n-1}\sum\limits_{k=1}^{n-i}\frac{(-\theta)^{k+1}(i-j)_{(j-2)}}{(\theta+i-j-1)_{(j-1)}(\theta+i-1)_{(k)}},
\end{multline}
while $\e_\theta(C^{\eta}_{n}(n))=(n-2)!/(\theta+2)_{(n-3)}$, and $\e_\theta(C^{\eta}_{n-1}(n))=0$. For $j\geq 2$, we have
\begin{equation}\label{limECj}
\begin{split}
&\lim_{n\longrightarrow \infty} \e(C^{\eta}_{j}(n))=\int_0^1\int_0^1\frac{\theta^2e^{-\theta y}x^{\theta-1}(1-x)^{j-2}(1-y)^{\theta+j-1}}{(1-x+xy)^{j-1}}dxdy\\[6pt]
&+\mathbbm{1}\{j\geq 3\}\frac{\theta^3(j-2)!}{\theta_{(j+1)}}\left(\theta+j-1-((\theta+j-1)^2+j-1)\int_0^1 e^{-\theta x}(1-x)^{\theta+j}dx\right)\\[6pt]
&-\mathbbm{1}\{j=2\}\frac{\theta^2}{\theta+1}\int_0^1 e^{-\theta x}(1-x)^{\theta+2}dx.
\end{split}
\end{equation}
\end{theorem}

\begin{proof}
Having $q_i=\theta/(\theta+i-1)$, for $i\geq 3$, (\ref{meanX}) gives the first part of the theorem. Also, as $nq_n\to \theta$, as $n\to \infty$, from Lemma~\ref{Thm:ECk},  $\lim_{n\to \infty}\e_\theta(C^{\eta}_{j}(n))$ exists and is finite. To find this limit, note that for $j\geq 2$
\begin{eqnarray*}
0\leq \lim_{n\to \infty}\sum_{i=n}^\infty\frac{\p_{\theta}(\eta_i^n=1)\theta(i-j)_{(j-2)}}{(\theta+i-j-1)_{(j-1)}}&\leq& \lim_{n\to \infty}\sum_{i=n}^\infty\p(\eta_i^n=1)q_{i-j}\\
&\leq& \lim_{n\to \infty}\sum_{i=n}^\infty q_iq_{i-j}=0.
\end{eqnarray*}
Thus, we get
\begin{equation}\label{Eq-lim-Cj}
\begin{split}
\lim_{n\longrightarrow \infty} &\e(C^{\eta}_{j}(n))=\sum_{i=j+1}^\infty\frac{\theta(i-j)_{(j-2)}}{(\theta+i-j-1)_{(j-1)}} \lim_{n\longrightarrow \infty} \p_{\theta}(\eta_i^n=1)\\[6pt]
&-\frac{\theta (j-1)!}{(\theta+1)_{(j-1)}} \lim_{n\longrightarrow \infty} \p_{\theta}(\eta_{j+2}^n=1)
-\frac{\theta(j-2)!}{\theta_{(j-1)}} \lim_{n\longrightarrow \infty} \p_{\theta}(\eta_{j+1}^n=1)\\[6pt]
&+\frac{(\theta+1)\theta(j-2)!}{\theta_{(j-1)}} \lim_{n\longrightarrow \infty} \p_{\theta}(\eta_{j+1}^n=1),
\end{split}
\end{equation}
where the second term on the right of the last equation is subtracted from the sum as $\p_\theta(\eta_{j+2}^n=\eta_{2}^n=1,\eta_{j+1}^n=\dots=\eta_{3}^n=0)=0$ while the two last terms appear because of the discrepancy between the first term of the sum (for $i=j+1$) and the value of $\p_\theta(\eta_{j+1}^n=\eta_{1}^n=1,\eta_{j+1}^n=\dots=\eta_{2}^n=0)=(\theta+1)\theta(j-2)!\p(\eta_{j+1}^n=1)/\theta_{(j-1)}$. To calculate (\ref{Eq-lim-Cj}), we first note that
\begin{eqnarray*}
\sum_{r=0}^\infty\frac{(r+1)_{(j-2)}}{(\theta+r)_{(j-1)}}(1-y)^r&=&\frac{(j-2)!}{\theta_{(j-1)}}\sum_{r=0}^\infty\frac{(j-1)_{(r)}\theta_{(r)}}{(\theta+j-1)_{(r)}r!}(1-y)^r\\
&=& \frac{\Gamma(j-1)\Gamma(\theta)}{\Gamma(\theta+j-1)} {}_2F_1(j-1,\theta;\theta+j-1;1-y)\\
&=&\int_0^1 x^{\theta-1}(1-x)^{j-2}(1-x(1-y))^{-(j-1)}dx,
\end{eqnarray*}
where the hypergeometric function $_2F_1(a,b;c;z)=\sum_{r=0}^\infty a_{(r)}b_{(r)}z^r/(c_{(r)}r!)$, and the last equality is given by the Euler type integral representation for $_2F_1$, for $Re(c)>Re(b)>0$. Having this, from Lemma~\ref{marginal-eta}, after interchanging the sum and integral, the first term on the right of (\ref{Eq-lim-Cj}) reduces to
\begin{equation*}
\begin{split}
\int_0^1&\theta^2 e^{-\theta y}(1-y)^{\theta+j-1}\sum_{r=0}^\infty\frac{(r+1)_{(j-2)}}{(\theta+r)_{(j-1)}}(1-y)^rdy\\[6pt]
&=\int_0^1\int_0^1\frac{\theta^2e^{-\theta y}x^{\theta-1}(1-x)^{j-2}(1-y)^{\theta+j-1}}{(1-x+xy)^{j-1}}dxdy.
\end{split}
\end{equation*}

Again from Lemma~\ref{marginal-eta}, the second, third and the fourth terms on the right of (\ref{Eq-lim-Cj}) equal
\begin{multline*}
\frac{\theta^3(j-2)!}{\theta_{(j)}}\left(-(j-1)\int_0^1 e^{-\theta u}(1-u)^{\theta+j}du+(\theta+j-1)\int_0^1 e^{-\theta u}(1-u)^{\theta+j-1}du\right)\\[6pt]
=\frac{\theta^3(j-2)!}{\theta_{(j)}}\left\lbrace-(j-1)\int_0^1 e^{-\theta u}(1-u)^{\theta+j}du\right.\\[6pt]
\left.-\left(\frac{\theta+j-1}{\theta+j}\right) \left(-1+\theta\int_0^1 e^{-\theta u}(1-u)^{\theta+j}du\right)\right\rbrace,
\end{multline*}
which completes the proof of this part for $j\geq 3$, after simplification. Similarly, for $j=2$, we can write
\begin{equation*}
\begin{split}
\lim_{n\longrightarrow \infty} \e(C^{\eta}_{j}(n))&=\sum_{i=j+1}^\infty\frac{\theta(i-j)_{(j-2)}}{(\theta+i-j-1)_{(j-1)}} \lim_{n\longrightarrow \infty} \p_{\theta}(\eta_i^n=1)\\
&-\frac{\theta}{\theta+1}\lim_{n\longrightarrow \infty} \p_{\theta}(\eta_4^n=1)
\end{split}
\end{equation*}
where the first term on the r.h.s. equals the first term of (\ref{Eq-lim-Cj}) and the last term on the r.h.s. equals 
$$
-\frac{\theta^2}{\theta+1}\int_0^1 e^{-\theta x}(1-x)^{\theta+2}dx.
$$
\end{proof}

Although very useful for numerical evaluations, the double integral in (\ref{limECj}) does not provide a simple expression for $\lim_{n\rightarrow \infty} \e_\theta\, C_j^\eta(n)$. The following result gives an approximation for the limit.

\begin{proposition}\label{prop:limECeta}
For any $j\geq 2$ and  $m\geq 1$, we have
\begin{multline}\label{ApproxlimECj}
\lim_{n\longrightarrow \infty} \e(C^{\eta}_{j}(n))= \mathbbm 1(j\geq 3) \frac{\theta(j-2)!}{(\theta+2)_{(j-3)}}\int_0^1 e^{-\theta x}(1-x)^{\theta+j-1} \ dx \\[4pt]
+\mathbbm{1}(j=2)\theta\int_0^1 e^{-\theta x}(1-x)^{\theta+1} \ dx+\sum\limits_{k=1}^{2m}(-1)^{k+1} \bar b_k{(\theta,j)}+\varepsilon(m,\theta,j),
\end{multline}
with $\varepsilon(m,\theta,j)\leq  \bar b_{2m+1}{(\theta,j)}$, where 
 \begin{equation*}
\begin{split}
\bar b_{k}{(\theta,j)}=& \frac{\theta^{k}(k-1)!(k(j-1)+\theta(k+j-1))}{(\theta+1)_{(k)}(j-1)_{(k+1)}}\\
&-\frac{\theta^{k}(j-2)!((k-1+(\theta+1)j)(\theta+j)-k)}{(\theta+1)_{(k+j)}}.
\end{split}
\end{equation*}
\end{proposition}
\begin{proof}
As $n\rightarrow\infty$, the first term on the right of (\ref{Eq-cor:ECK}) converges to $0$ and the second term converges to the first and the second terms on the right of (\ref{ApproxlimECj}), for $j\geq 3$ and $j=2$, respectively.  Let
\[
b_k(n,\theta,j)=\sum\limits_{i=j+3}^{n-k} \frac{\theta^{k+1} (i-j)_{(j-2)}}{(\theta+i-j-1)_{(j-1)}(\theta+i-1)_{(k)}},
\]
and note that $\lim_{n\rightarrow \infty} b_k(n,\theta,j)=\bar{b}_k(\theta,j)$. From Lemma \ref{Thm:ECk}, the limit of the double sum on the r.h.s. of (\ref{Eq-cor:ECK}) can be written
\[
\begin{split}
\lim\limits_{n\rightarrow \infty} \sum\limits_{k=1}^{n-j-3} (-1)^{k+1} b_k(n,\theta,j)&=\lim\limits_{n\rightarrow \infty} \sum\limits_{k=1}^\infty (-1)^{k+1} b_k(n,\theta,j)\\
&= \sum\limits_{k=1}^{2m} (-1)^{k+1} \bar b_k(\theta,j) +\varepsilon(m,\theta,j),
\end{split}
\]
where 
\[
\begin{split}
0\leq \varepsilon(m,\theta,j)=\lim\limits_{n\rightarrow \infty} \sum\limits_{k=2m+1}^\infty (-1)^k b_k(n,\theta,j)\leq \lim\limits_{n\rightarrow \infty} b_{2m+1}(n,\theta,j).
\end{split}
\]
\end{proof}

As an example,  Table~\ref{tab:limECk} provides the approximation for the $\e(C^{\eta}_{j}(n))$ as $n$ tends to infinity, for $2\leq j\leq 7$ and $m=2$.  The numerical results perfectly match with the exact values derived from (\ref{limECj}). The values are compared to their counterpart for the classical Feller coupling $\tilde{\xi}$, $\lim\limits_{n\longrightarrow \infty}\e({C}^{\tilde{\xi}}_j(n))=\theta/j$.

\begin{table}[H]
\begin{center}
\caption{The approximation of the $\e C_j^{\eta}(n)$ as $n$ tends to infinity for $\theta=0.5$, $2\leq j\leq 7$, by applying Proposition~\ref{prop:limECeta} for $m=2$.}\label{tab:limECk}
\begin{tabular}{clcl}
\hline
$j$  &$\lim\limits_{n\longrightarrow \infty} \e C_j^{\eta}(n)$  & Error &$\theta/j$\rule{0pt}{2.6ex}\rule[-1.2ex]{0pt}{0pt}\\
& approx. & & \\
\hline
\rule{0pt}{3ex}
$2$ &$0.255318$ &$9.86668\times 10^{-7}$&$0.250$\\
$3$ &$0.19468$ &$4.38404\times 10^{-7}$&$0.167$\\
$4$&$0.137891$ &$2.20947\times 10^{-7}$&$0.125$\\
$5$&$0.107192$ &$1.21856\times 10^{-7}$&$0.100$\\
$6$&$0.0878281$ &$7.19514\times 10^{-8}$&$0.083$\\
$7$&$0.0744583$ &$4.48278\times 10^{-8}$&$0.072$\\\hline
\end{tabular}
\end{center}
\end{table}
Although the details are omitted, one can also easily obtain the variance
\[
\begin{split}
{\rm Var}(C_{j}(n))&=\sum\limits_{i=j+1}^{n+1} \mathcal R^{(n)}_{j,i}(1-\mathcal R^{(n)}_{j,i}) +2 \sum\limits_{l\leq i-j} \mathcal R^{(n)}_{j,i}\mathcal R^{(i-j-1)}_{j,l}-2 \sum\limits_{l<i} \mathcal R^{(n)}_{j,i}\mathcal R^{(n)}_{j,l}.
\end{split}
\]

Some numerical examples of the variance for the $\eta$ process are displayed in Table~{\ref{tab:VCj}.

\begin{table}
\begin{center}
\caption{The variance of $C_j^{\eta}(n)$ for $\theta=0.5$, $3\leq j\leq 7$.}\label{tab:VCj}
\begin{tabular}{clcc}
\hline
$j$  &$n=20$  &$n=50$&$n=100$ \rule{0pt}{2.6ex}\rule[-1.2ex]{0pt}{0pt}\\
\hline
\rule{0pt}{3ex}
$3$ &$0.185732$ &$0.177823 $ &$0.175253$ \\
$4$& $0.142278$ &$0.133938$ &$0.131308$\\
$5$&$0.116493$ &$0.107688$ &$0.104996$\\
$6$&$0.0996403$&$0.090335$ &$0.087578$\\
$7$&$0.087877$ &$0.078045$ &$0.075221$\\\hline
\end{tabular}
\end{center}
\end{table}

\section{Conditioning on generalized Feller coupling}\label{sec:coupling}
In this section we establish necessary and sufficient conditions for a conditional relation between the derangement Markov chains $X^{n,\mathbf{p}}$ and $(Y_n,\cdots, Y_1)$, for $n\in \N$.  We explore some properties of the Generalized Feller Coupling (GFC) $Y^{\bm\theta}=(Y^{\bm\theta})_{i=1}^\infty$,  along with some of its applications for $X^{n,\textbf{p}}$ via the conditional relation.  We also obtain the weak limit of $X^{n,\mathbf{p}}$, as $n\rightarrow \infty$, as a $\{0,1\}$-valued infinite Markov chain, and provide a conditional relation between the limit and the infinite sequence $Y^{\bm\theta}$. We also apply the theory developed in this section to the specific example of the playground process $\eta$.


\subsection{A conditional relation}\label{sec:conditional}
~\cite{sjt2021} constructs a Markov chain $\tilde\eta_n,\cdots, \tilde\eta_1$ generating the cycle counts of the random derangement sampled from ESF$_n(\theta)$ conditioned on not having fixed points. In other words, 
\begin{equation}\label{FClaw}
\mcL(C^{\tilde\eta}_2(n),\ldots,C^{\tilde\eta}_n(n)) = \mcL(C^{\tilde\xi}_2(n),\ldots,C^{\tilde\xi}_n(n) \mid C^{\tilde\xi}_1(n) = 0),
\end{equation}
where as before,  $\tilde\xi$ denotes the classical Feller coupling.  Studying a Markov chain is not always easy if one directly uses its transition probabilities, while using the conditional relations such as (\ref{FClaw}) sometimes makes computations easier.  Motivated by this, one can ask if there exists an infinite sequence of independent $\{0,1\}$-valued random variables $Y_1=1, Y_2, \cdots$,  for which the law of $(Y_i)_{i=1}^n$ conditional on having no $11$ patterns in $1,Y_n,\cdots,Y_1=1$, coincides with that of the Markov chain $X_n^n,...,X_1^n$.  More specifically,   we define a generalized Feller coupling (GFC) as a sequence of independent Bernoulli $\{0,1\}$-valued random variables $Y^{\bm{\theta}}=Y :=(Y_i)_{i=1}^\infty$ such that 
\begin{equation*}\label{probxi}
\p_{\bm{\theta}}(Y_i=0)=\frac{i-1}{i-1+\theta_{i}},
\end{equation*}
where $\bm{\theta}=(\theta_i)_{i\in \N}$ is a sequence of strictly positive real numbers. In this paper, we always assume that $\theta_1=1$.  For $j\in \N$, let 
\begin{equation}\label{Deltaj}
\Delta_j:=\{(a_j,\ldots,a_1)\in \{0,1\}^j: a_1=1, a_j+\sum\limits_{i=1}^{j-1} a_i a_{i+1}=0\},
\end{equation}
and let $Y^{n,\bm{\theta}}=Y^n :=(Y_n,\cdots, Y_1)$.   Theorem~\ref{irvgeneral} gives the necessary and sufficient conditions under which
\begin{equation}\label{condp}
\p_{\mathbf{p}}(X_n^n=a_n,...,X_1^n=a_1)=\p_{\bm \theta}(Y_n=a_n,...,Y_1=a_1|Y^n\in\Delta_n),
\end{equation}
for $(a_n,...,a_1)\in \Delta_n$. In other words,  we investigate the necessary and sufficient conditions under which the cycle counts $(C_2(n),\ldots,C_n(n))$  of the permutation generated by $X^{n,\mathbf{p}}$ have a distribution determined by
\begin{equation}\label{GFClaw}
\mcL(C_2(n),\ldots,C_n(n)) = \mcL(\tilde C_2(n),\ldots,\tilde C_n(n) \mid \tilde C_1(n) = 0),
\end{equation}
where $(\tilde C_2(n),\ldots,\tilde C_n(n))$ is the cycle counts of $Y_n,...,Y_1=1$.  Furthermore,  having a Markov chain $X^{n,\mathbf{p}}$, we can find a sequence $\bm\theta$ such that (\ref{condp}) holds,  and vice versa, having a sequence of independent random variables $Y^{n,\bm\theta}$, we can find a sequence $\mathbf{p}$ such that (\ref{condp}) holds.  To make this precise,  let $\gamma_i(\bm\theta)=\p_{\bm\theta}((Y_i,...,Y_1)\in\Delta_i)$ and note that
\begin{equation}\label{rec2}
\gamma_i(\bm\theta)=\frac{i-1}{i-1+\theta_i}\left(\gamma_{i-1}(\bm\theta)+\frac{\theta_{i-1}}{i-2+\theta_{i-1}}\gamma_{i-2}(\bm\theta)\right),
\end{equation}
for $i\geq 3$, with initial conditions $\gamma_1=0$ and $\gamma_2=\p(Y_2=0)=1/(1+\theta_2)$.  Let 
$$
\bm\theta_{<n>}:=\theta_1(\theta_2+1)(\theta_3+2)\cdots(\theta_{n}+n-1),
$$
and recall we assume $\theta_1=1$ in this paper. The next proposition gives an exact formula for $\gamma(\bm\theta)$.

\begin{proposition}\label{prop:gammai}
For $i\geq 2$
\begin{equation*}\label{gammai_theta}
\gamma_i(\bm\theta)= G_{i-1}(\bm\theta)\frac{(i-1)!}{\bm\theta_{<i>}},
\end{equation*}
where $G_0(\bm\theta)=0$ and $G_1(\bm\theta)=G_2(\bm\theta)=1$,
$$
G_i(\bm\theta):=1+\sum_{k=1}^{\lfloor \frac{i-1}{2}\rfloor}\sum_{(i_1,\dots,i_k)}\frac{\theta_{i_1}\cdots\theta_{i_k}}{(i_1-1)\cdots(i_k-1)},
$$
for $i\geq 3$, and the last sum is over all $(i_1,\dots,i_k)$ such that $2<i_j\leq i$ and $i_{j+1}>i_j+1$.
\end{proposition}

\begin{proof}
Note that $G_0(\bm\theta)=\gamma_1(\bm\theta)=0$ and $G_1(\bm\theta)/\bm\theta_{<2>}=1/(1+\theta_2)=\gamma_2(\bm\theta)$.  For any $i\geq 3$,  
\begin{eqnarray*}
\lefteqn{\frac{i-1}{i-1+\theta_{i}}\left(\frac{(i-2)!}{\bm\theta_{<i-1>}}G_{i-2}(\bm\theta)+\frac{\theta_{i-1}}{i-2+\theta_{i-1}}\frac{(i-3)!}{\bm\theta_{<i-2>}}G_{i-3}(\bm\theta)\right)}\\[4pt]
&=&\frac{(i-1)!}{\bm\theta_{<i>}}\left(G_{i-2}(\bm\theta)+\frac{\theta_{i-1}}{i-2} G_{i-3}(\bm\theta)\right)=\frac{(i-1)!}{\bm\theta_{<i>}}G_{i-1}(\bm\theta),
\end{eqnarray*}
 where the last equality follows from the definition of $G_i(\bm\theta)$. Hence for any $i$,  $G_{i-1}(\bm\theta)(i-1)!/\bm\theta_{<i>}$ satisfies (\ref{rec2}),  hence the result.
\end{proof}

Now we are ready to state the main theorem of this section as follows. 

\begin{theorem}\label{irvgeneral}
For any $4\leq n\in \N$,  the following are equivalent.
\begin{itemize}
\item[(i)] For any $(a_n,\cdots,a_1)\in \Delta_n$, $$
\p_{\mathbf{p}}(X_n^n=a_n,...,X_1^n=a_1)=\p_{\bm\theta}(Y_n=a_n,...,Y_1=a_1|Y^n\in\Delta_n).$$
\item[(ii)] $
\displaystyle p_i=\frac{(i-1+\theta_i)\gamma_i(\bm\theta)}{(i-1+\theta_i)\gamma_i(\bm\theta)+\theta_i\gamma_{i-1}(\bm\theta)}=\frac{G_{i-1}(\bm\theta)}{G_i(\bm\theta)},
$ for $i=3,...,n-1$.
\item[(iii)] $
\displaystyle \theta_i=\frac{(i-1) q_{i}}{p_{i}p_{i-1}},
$ for $i=3,...,n-1$.
\end{itemize}
\end{theorem}

\begin{proof}
First note that, from (\ref{rec2}) and Proposition~\ref{prop:gammai},  the last equality in $(ii)$ holds for $i\geq 3$.
The same lines of argument as those in~\cite{sjt2021} proves $(i)\Rightarrow(ii)$. More precisely, suppose $(i)$ holds, then for $(r_n,\dots,r_{i+2},0,1)\in\Delta_{n-i+1}$
\begin{eqnarray*}
\lefteqn{p_i=\p_{\mathbf{p}}(X_i=0\mid X_{i+1}=0)=\p_{\mathbf{p}}(X_i= X_{i+1}=0)/\p_{\mathbf{p}}(X_{i+1}=0)}\\[6pt]
&=&\frac{\gamma_n^{-1}(\bm\theta)\p_{\bm\theta}(Y_n=r_n,\dots,Y_{i+2}=r_{i+2},Y_{i+1}=0)\gamma_i(\bm\theta)}{\gamma_n^{-1}(\bm\theta)\p_{\bm\theta}(Y_n=r_n,\dots,Y_{i+2}=r_{i+2},Y_{i+1}=0)(\gamma_i(\bm\theta)+\p(Y_i=1)\gamma_{i-1}(\bm\theta))}\\[6pt]
&=&\frac{(i-1+\theta_i)\gamma_i(\bm\theta)}{(i-1+\theta_i)\gamma_i(\bm\theta)+\theta_i\gamma_{i-1}(\bm\theta)}.
\end{eqnarray*}

To show  $(ii)\Rightarrow(i)$, for $n\in \N$,  suppose that
\begin{equation*}
\begin{split}
    \p(Y_m=a_m,...,Y_1=a_1|Y^m\in\Delta_m)=& \p(X_m^m=a_m,...,X_1^m=a_1)\\
    =&\p(X_m^n=a_m,...,X_1^n=a_1\mid X_{m+1}^n=1)
\end{split}
\end{equation*}
holds for any $(a_m,\cdots,a_1)\in \Delta_m$ and $m<n$.  Let $(a_n,\cdots,a_1)\in \Delta_n$ be such that $\prod_{i=2}^n (1-a_i)=0$, which means there exists at least one index $2<i<n$ s.t. $a_i=1$. Let $1<j<n-1$ be the largest index for which $a_{j+1}=1$. Then $(a_j,\ldots,a_1)\in \Delta_{j}$ and
\begin{eqnarray*}
\lefteqn{
\p(X_n^n=a_n,...,X_1^n=a_1)}\\
&=& \p(X_{j}^n=0,X_{j-1}^n=a_{j-1},...,X_1^n=a_1\mid X_{j+1}^n=1)\p(X_{n}^n=a_n,..., X_{j+1}^n=1) \\
&=&\frac{\p(Y_j=0,Y_{j-1}=a_{j-1},...,Y_1=a_1)}{\gamma_j(\bm\theta)}q_{j+1}\prod_{i=j+2}^{n-1}p_i\\
&=&  \frac{\p(Y_j=0,Y_{j-1}=a_{j-1},...,Y_1=a_1)}{\gamma_j(\bm\theta)}\\
&& \times\frac{\p(Y_{j+2}=0,Y_{j+1}=1)\gamma_j(\bm\theta)}{\gamma_{j+2}(\bm\theta)}\prod_{i=j+2}^{n-1}\frac{\p(Y_{i+1}=0)\gamma_i(\bm\theta)}{\gamma_{i+1}(\bm\theta)}\\
&=& \p(Y_n=a_n,...,Y_1=a_1|Y^{n}\in\Delta_n),
\end{eqnarray*}
where in the last-but-one line we used (\ref{rec2}). In the case that $a_i=0$,  for $2\leq i \leq n$,  we have 
\begin{eqnarray*}
\p(X_n^n=a_n,...,X_1^n=a_1)&=&\prod_{i=2}^{n-1}p_i = \prod_{i=2}^{n-1}\frac{\p(Y_{i+1}=0)\gamma_i(\bm\theta)}{\gamma_{i+1}(\bm\theta)}\\
&=& \p(Y_n=a_n,...,Y_1=a_1|Y^{n}\in\Delta_n),
\end{eqnarray*}
since $\gamma_2=\p(Y_2=0)$ and $\p(Y_1=1)=1$, hence $(ii)\Rightarrow (i)$ for $n\in \N$.

Note that $(iii)$ is straightforward from $(ii)$. More precisely, using (\ref{rec2}) and $(ii)$, for $i=3,\cdots, n-1$, we get
\begin{equation}\label{pq}
    p_i=\frac{i\gamma_i}{(i+\theta_{i+1})\gamma_{i+1}}, \ \ q_i=\frac{i\theta_i\gamma_{i-1}}{(i+\theta_{i+1})(i-1+\theta_i)\gamma_{i+1}}, 
\end{equation}
where substituting these into $(i-1)q_i/(p_ip_{i-1})$ leads to the equation in $(iii)$.
To prove $(iii)\Rightarrow (ii)$, we first show if $(iii)$ holds, we have
\begin{equation}\label{gammai}
\gamma_i(\bm\theta)=\frac{p_i}{1+\theta_2}\prod\limits_{j=2}^{i-1}\frac{p_j}{p_j p_{j+1}+q_{j+1}},
\end{equation}
for $i=3,...,n-1$. To establish this, denote, temporarily, by $\tilde \gamma_i(\bm\theta)$ the r.h.s. of (\ref{gammai}). We show the $\tilde\gamma_i(\bm\theta)$ satisfies (\ref{rec2}), hence $\tilde\gamma_i(\bm\theta)=\gamma_i(\bm\theta)$. It is easy to see $\tilde\gamma_i(\bm\theta)=\gamma_i(\bm\theta)$, for $i=3,4$, just from the definition, so the calculation is omitted here. Now for $i\geq 4$, substituting $\theta_i=(i-1)q_i/(p_ip_{i-1})$ and $\theta_{i+1}=i q_{i+1}/(p_{i+1}p_{i})$, and simplifying, we get
\begin{equation*}
\begin{split}
\frac{i}{i+\theta_{i+1}}(\tilde\gamma_i(\bm\theta)+&\frac{\theta_i}{i-1+\theta_i}\tilde\gamma_{i-1}(\bm\theta))=\frac{p_{i+1}p_i}{p_{i+1}p_i+q_{i+1}}(\tilde\gamma_i(\bm\theta)+\frac{q_i}{p_ip_{i-1}+q_i}\tilde\gamma_{i-1}(\bm\theta))\\
&=\frac{p_{i+1}p_i}{p_{i+1}p_i+q_{i+1}} \tilde\gamma_{i-1}(\bm\theta)\left(\frac{p_i}{p_{i-1}}.\frac{p_{i-1}}{p_{i-1}p_i+q_i}+\frac{q_i}{p_i p_{i-1}+q_i}\right)\\
&=\frac{p_{i+1}p_i}{(p_{i+1}p_i+q_{i+1})(p_i p_{i-1}+q_i)}\tilde\gamma_{i-1}(\bm\theta)=\tilde\gamma_{i+1}(\bm\theta), 
\end{split}
\end{equation*}
as claimed, hence $\tilde\gamma_{i}=\gamma_{i}$, for $i=3,\cdots,n-1$, therefore (\ref{gammai}) holds if we assume $(iii)$. Thus, applying (\ref{gammai}) and $\theta_i=(i-1)q_i/(p_ip_{i-1})$, the middle term in $(ii)$ simplifies to $p_i$, hence $(iii)\Rightarrow (ii)$.
\end{proof}
As discussed before, the last theorem gives a simple way to construct $Y^{n,\bm{\theta}}$ from $X^{n,\mathbf{p}}$, and vice versa. Note also that (\ref{gammai}) provides an interesting way to compute the $\gamma_i(\bm\theta)$, that does not involve usual inclusion-exclusion arguments for such quantities. 
We now apply Theorem~\ref{irvgeneral} to establish the main conditioning result for the playground game $\eta$. Denote by
$$
B(\tilde{z}_1,\tilde{z}_2)=\frac{\Gamma(\tilde z_1)\Gamma(\tilde z_2)}{\Gamma(\tilde z_1+\tilde z_2)}=\int_0^1 u^{\tilde{z}_1-1} (1-u)^{\tilde{z}_2-1}du,
$$
the Beta function for complex variables $\tilde{z}_1, \tilde{z}_2$ with $Re(\tilde{z}_1),Re(\tilde{z}_2)>0$.
\begin{corollary}\label{PGC}
For given $0<\theta_1^*,\theta_2^*\leq 1$, there exists a unique sequence of independent Bernoulli $\{0,1\}$-valued random variables $\xi:=(\xi_i)_{i=1}^\infty$, with
\begin{equation*}\label{probxit}
\p_\theta(\xi_i=0)=\frac{i-1}{i-1+\theta_i^*},
\end{equation*}
where $\theta_i^*= \theta(\frac{\theta}{i-2}+1)$, for $i\geq 4$, and $\theta_3^*=\theta$, such that for any $4\leq n\in \N$ and  $(a_n,...,a_1)\in\Delta_n$, letting $\xi^n=(\xi_n,\cdots,\xi_1)$, we have
\begin{equation}\label{condlawplayground}
\p_\theta(\eta_n^n=a_n,...,\eta_1^n=a_1)=\p_\theta(\xi_n=a_n,...,\xi_1=a_1|\xi^{n}\in\Delta_n).
\end{equation}
In addition, for $3\leq n\in \N$
\begin{equation}\label{deltai}
\delta_n(\theta):=\p_\theta(\xi^n\in\Delta_n)=\frac{(n-1)!(\theta^2+\theta+2) (\theta+2)_{(n-3)}}{(1+\theta^*_2)(\theta+2)\prod\limits_{k=1}^{n-2}(k(k+1)+\theta(\theta+k))},
\end{equation}
while $\delta_1=0$ and $\delta_2=\p(\xi_2=0)=1/(1+\theta^*_2)$
. Also,

$$
\lim_{n\to \infty}\delta_n(\theta)=\frac{\theta^2+\theta+2}{1+\theta^*_2} B(z_1(\theta),z_2(\theta)),
$$
where $z_1(\theta)=\frac{3+\theta-\sqrt{(1-\theta)(1+3\theta)}}{2}$, $z_2(\theta)=\frac{3+\theta+\sqrt{(1-\theta)(1+3\theta)}}{2}$.

\end{corollary}
\begin{proof}
First note that from Theorem~\ref{irvgeneral}, for any $n\geq 4$, (\ref{condlawplayground}) holds if and only if $\theta^*_i=(i-1)q_i/(p_ip_{i-1})$, for $3\leq i\leq n-1$, which letting $q_i=1-p_i=\theta/(\theta+i-1)$ and $p_2=1$, simplifies to the values of $\theta^*_i$ given in the statement of the corollary, and results in (\ref{deltai}), if we substitute them into (\ref{gammai}) for $i\geq 3$. Note that $\delta_1(\theta)=0, \delta_2(\theta)=1/(1+\theta^*_2)$ are obvious from the definition.

For the limit, as $k(k+1)+\theta(\theta+k)=(k-1+z_1(\theta))(k-1+z_2(\theta))$ and $z_1(\theta)+z_2(\theta)=\theta+3$, writing $x_{(m)}=\Gamma(x+m)/\Gamma(x)$ concludes
\begin{equation*}
\begin{split}
\delta_n(\theta)=&\frac{(\theta^2+\theta+2)\Gamma(n)(\theta+3)_{(n-4)}}{(1+\theta^*_2)(z_1(\theta))_{(n-2)}(z_2(\theta))_{(n-2)}}\\
=&\frac{(\theta^2+\theta+2)B(z_1(\theta),z_2(\theta))}{1+\theta^*_2}\cdot\frac{\Gamma(n)\Gamma(n+\theta-1)}{\Gamma(n+z_1(\theta)-2)\Gamma(n+z_2(\theta)-2)},
\end{split}
\end{equation*}
where for large $n$, the term
\[
\frac{\Gamma(n)\Gamma(n+\theta-1)}{\Gamma(n+z_1(\theta)-2)\Gamma(n+z_2(\theta)-2)}\sim n^{2-z_1(\theta)}n^{\theta-1+2-z_2(\theta)}=n^{\theta+3-z_1(\theta)-z_2(\theta)}=1,
\]
as, once again, $z_1(\theta)+z_2(\theta)=\theta+3$, completing the proof.
\end{proof}

 
\subsection{Number of cycles revisited}\label{sec:cyclesRevisited}

For the classical Feller coupling $\tilde\xi_n^\theta,\cdots,\tilde\xi_1^\theta$, the probability generating function (pgf) of the number of cycles $K_n^{\tilde\xi}$ is given by
$$
\e_\theta(s^{K_n^{\tilde\xi}})=\frac{(\theta s)_{(n)}}{\theta_{(n)}}.
$$
The above pgf indeed has an interesting relation with the pgf of the number of cycles $K_n^{\tilde\eta}$ of the $\tilde\eta$ process~\cite{sjt2021}, that is

$$
\e_\theta(s^{K_n^{\tilde\eta}})=\frac{\lambda_n(\theta s)}{\lambda_n(\theta)}\e_\theta(s^{K_n^{\tilde\xi}}),
$$
where $\lambda_n(\theta)=\p_\theta((\tilde\xi_n,\cdots,\tilde\xi_1) \in \Delta_n)$ represents the probability that a $\theta$-biased random permutation of size $n$ is a derangement.   A similar relation holds for the total number of cycles of $X^{n,\mathbf{p}}$ and $Y^{n,\bm{\theta}}$.  To see this, denote by $K_n=K_n^{\mathbf{p}}$ and $\tilde K_n=\tilde K_n^{\bm{\theta}}$ the total number of cycles of $X^{n,\mathbf{p}}$ and $Y^{n,\bm{\theta}}$. We have
\begin{equation}\label{probKY}
\p_{\bm{\theta}}(\tilde K_n=k)=\sum\limits_{1=i_1<\cdots<i_k\leq n}\frac{(n-1)!\theta_{i_1}\cdots\theta_{i_k}}{(i_2-1)\cdots(i_k-1)\bm\theta_{<n>}},
\end{equation}
where $\p_\theta(\tilde K_n=1)=\theta_1(n-1)!/\theta_{<n>}$. Note that $\tilde K_n$ has the Poisson-binomial distribution.  The pgf of $\tilde K_n$ is given by
\begin{eqnarray*}
\e_{\bm{\theta}}(s^{\tilde K_n})&=& \sum\limits_{k=1}^n s^k \p_{\bm{\theta}}(\tilde K_n=k)\\
&=& \sum\limits_{k=1}^n s^k \sum\limits_{1=i_1<\cdots<i_k\leq n}\frac{(n-1)!\theta_{i_1}\cdots\theta_{i_k}}{(i_2-1)\cdots(i_k-1)\bm\theta_{<n>}}\\
&=& \frac{(s\bm\theta)_{<n>}}{\bm\theta_{<n>}}\sum\limits_{k=1}^n  \sum\limits_{i_1<\cdots<i_k}\frac{(n-1)!(s\theta_{i_1})\cdots(s\theta_{i_k})}{(i_2-1)\cdots(i_k-1)(s\bm\theta)_{<n>}}\\
&=& \frac{(s\bm\theta)_{<n>}}{\bm\theta_{<n>}}.
\end{eqnarray*}
Similarly,  for the number $K_n$ of cycles of $X^{n,\mathbf p}$ we have
\begin{equation}\label{probKX}
\p_{\mathbf{p}}(K_n=k)=\sum_{
\substack{1=i_1<\cdots<i_k< n \\ i_j+1<i_{j+1}}
}\frac{(n-1)!\theta_{i_1}\cdots\theta_{i_k}}{(i_2-1)\cdots(i_k-1)\gamma_n(\bm\theta)\bm\theta_{<n>}},
\end{equation}

Therefore, we obtain the following relation for the probability generating functions of $\tilde K_n$ and $K_n$.
\begin{theorem}
Suppose $
\displaystyle \theta_i=(i-1) q_{i}/(p_{i}p_{i-1}),$ for $i=3,...,n-1$. Then
$$
\e_{\mathbf{p}}(s^{K_n}) = \frac{\gamma_n(s\bm\theta)}{\gamma_n(\bm\theta)} \e_{\bm{\theta}}(s^{\tilde K_n}).
$$

\end{theorem}

\begin{proof}
As in the above discussion, we have
\begin{eqnarray*}
 \e_{\mathbf{p}}(s^{K_n})&=&\frac{\gamma_n(s\bm\theta)(s\bm\theta)_{<n>}}{\gamma_n(\bm\theta)\bm\theta_{<n>}}\sum_{k=1}^{\lfloor\frac{n}{2}\rfloor}\sum_{
\substack{1=i_1<\cdots<i_k< n \\ i_j+1<i_{j+1}}
}\frac{(n-1)!(s\theta_{i_1})\cdots(s\theta_{i_k})}{(i_2-1)\cdots(i_k-1)\gamma_n(s\bm\theta)(s\bm\theta)_{<n>}}\\[4pt]
&=&\frac{\gamma_n(s\bm\theta)(s\bm\theta)_{<n>}}{\gamma_n(\bm\theta)\bm\theta_{<n>}}=\frac{\gamma_n(s\bm\theta)}{\gamma_n(\bm\theta)} \e_{\bm{\theta}}(s^{\tilde K_n}).
\end{eqnarray*}
\end{proof}

\subsection{Cycle counts}\label{sec:cycleCounts}
In this section, we find the distribution of the cycle counts for the  GFC.  In fact, this is the counterpart of the Ewens sampling formula when $\theta$ is replaced by $\bm \theta=(\theta_i)_{i=1}^\infty$, hence more involved.  For a vector $r\in \{0,1\}^n$ with $r_1=1$,
\[\p_{\bm\theta}(Y^{n}=r)=\frac{(n-1)!}{\bm\theta_{<n>}}\theta_1\prod\limits_{1<i\leq n: r_i=1}\frac{\theta_{i}}{i-1}.\]
Letting $\prod_{i=1}^0x_i=1$, by convention, this is equivalent to
\begin{equation}\label{drnY}
\p_{\bm\theta}(Y^{n}=r)=\frac{(n-1)!}{\bm\theta_{<n>}}\theta_1\prod\limits_{i=1}^{\|r\|-1}\frac{\theta_{n+1-a_1-\cdots-a_i}}{n-a_1-\cdots-a_i},
\end{equation}
where $\|r\|$ is the number of cycles of $r$ and $a_i$ is the size of the $i$-th cycle (in the order of formation of the cycles).  In other words, $a_1=n+1-\max\{i\leq n; r_i=1\}$,  $a_2=n+1-a_1-\max\{i\leq n-a_1; r_i=1\}$, and so on.  


To compute $\p(\tilde C_1=c_1,\cdots,\tilde C_n=c_n)$ we sum  $\p_{\bm\theta}(Y^{n}=r)$  over all possible $r\in \{0,1\}^n$, $r_1=1$, which have the cycle type $(c_1,\cdots,c_n)$.  To this end,  for any $c=(c_1,...,c_n)\in \Z_+^n$ satisfying $\sum_{i=1}^nic_i=n$,  let $\|c\|=\sum_{i=1}^n c_i$. We define $\bar{c}=(\bar{c}_1,...,\bar{c}_{\|c\|})$ as follows. Let $i_1,...i_k$ be all numbers in $\{1,...,n\}$ such that $c_{i_j}\neq 0$. For $0<l \leq c_{i_1}$, let $\bar{c}_l=i_1$. Similarly, for any $m=2,...,k$, if $\sum_{j=1}^{m-1} c_{i_j} <l\leq \sum_{j=1}^{m} c_{i_j}$, let $\bar{c}_l=i_m$. For example, if $c=(2,0,1,0,4)$, then $\bar{c}=(1,1,3,5,5,5,5)$. Denote by $S_l$ the permutation group of size $l\in \N$.


\begin{proposition}\label{jointdrbnX}
For any $n\in \N$ and $c=(c_1,...,c_n)\in \Z_+^n$ with $\sum_{i=1}^n i c_i=n$,
\begin{equation}\label{jointDistY}
\p_{\bm\theta}(\tilde C_1(n)=c_1,\ldots,\tilde C_n(n)=c_n)=\frac{(n-1)!}{\bm\theta_{<n>}}\theta_1\prod\limits_{i=1}^n\frac{1}{c_i!}\sum\limits_{\sigma\in S_{\|c\|}}\prod\limits_{i=1}^{\|c\|-1} \frac{\theta_{\epsilon(i,c,\sigma)}}{\epsilon(i,c,\sigma)-1},
\end{equation}
where $\epsilon(i,c,\sigma)=\epsilon_n(i,c,\sigma):=n+1-\sum_{j=1}^i \bar{c}_{\sigma(j)}$, for $i=1,\cdots,\|c\|$. Furthermore, if $\theta_i=(i-1)q_i/(p_ip_{i-1}), \ i\geq 3$, then for any $2\leq n\in \N$ and any $c=(c_2,\cdots,c_n)$ with $\sum_{i=2}^n i c_i=n$, we have
\begin{equation}\label{jointDistX}
\begin{split}
\p_{\mathbf{p}}(C_2(n)=c_2,\ldots,C_n(n)=c_n)&=  \frac{(n-1)!\theta_1}{\gamma_n(\bm\theta)\bm\theta_{<n>}}\prod\limits_{i=2}^n\frac{1}{c_i!}\sum\limits_{\sigma\in S_{\|c\|}}\prod\limits_{i=1}^{\|c\|-1} \frac{\theta_{\epsilon(i,c,\sigma)}}{\epsilon(i,c,\sigma)-1} \\
&=  \frac{1}{c_n! }\prod_{i=2}^{n-1}\frac{p_i}{c_i!}\sum\limits_{\sigma\in S_{\|c\|}}\prod\limits_{i=1}^{\|c\|-1} \frac{q_{\epsilon(i,c,\sigma)}}{p_{\epsilon(i,c,\sigma)}p_{\epsilon(i,c,\sigma)-1}}.
\end{split}
\end{equation}
\end{proposition}
\begin{proof}
The first part is straightforward from (\ref{drnY}), by summing over all possible $r\in \{0,1\}^n$, with $r_1=1$. The first equality in (\ref{jointDistX}) follows from (\ref{jointDistY}) and the conditional relation given in Theorem~\ref{irvgeneral}, while the second equality comes from the relation of $\bm\theta$ and $\mathbf{p}$, and (\ref{gammai}).
\end{proof}
Let by convention $\epsilon(0,c,\sigma)=\infty$. Note that $\epsilon(\|c\|,c,\sigma)=1$ and $\epsilon(\|c\|-1,c,\sigma)\geq 3$, for any $\sigma$ and $c$. The next corollary follows immediately.
\begin{corollary}
For any $n\geq 3$ and $c=(c_2,...,c_n)\in\Z_+^{n-1}$, with $\sum_{i=2}^n ic_i=n$,
\begin{multline*}
\p(C^\eta_2(n)=c_2,\ldots,C^\eta_n(n)=c_n)=\\
\frac{(\theta+1)(n-2)!\theta^{\|c\|}}{\theta_{(n-1)}}\prod\limits_{i=2}^n \frac{1}{c_i!}\sum\limits_{\sigma\in S_{\|c\|}} \epsilon^*(c,\sigma) \prod\limits_{i=1}^{\|c\|-1} \frac{\theta+\epsilon(i,c,\sigma)-2}{(\epsilon(i,c,\sigma)-1)(\epsilon(i,c,\sigma)-2)},
\end{multline*}
where $\epsilon(i,c,\sigma)$ is defined as in Proposition~\ref{jointdrbnX}, and 
\[
\epsilon^*(c,\sigma)=\mathbbm{1}\{\epsilon(\|c\|-1,c,\sigma)>3\}+\frac{1}{\theta+1}\mathbbm{1}\{\epsilon(\|c\|-1,c,\sigma)=3\}.
\]
\end{corollary}

\subsection{The weak limit of $X^{n,\mathbf{p}}$}
As $X_m^m=0$, for any $m>2$, one cannot recover the outcomes or random derangements of $X^n$  from those of $X^m$, $n>m$. In other words, one cannot obtain the law of $X^m$ from the law of $X^n$ by projecting on the first $m$ components, i.e.
$$
\tilde\mu_n\pi^{-1}_m\neq \tilde\mu_m, \ n>m\geq 2,
$$
where $\tilde\mu_m$ is the law of $X^m$, $\pi_m(x)=(x_i)_{i=1}^m$ for $x\in\{0,1\}^\N$ or $x\in\{0,1\}^n$, $n\geq m$, and finally $\tilde \mu_n\pi^{-1}_m$ is the image of $\tilde \mu_n$ under $\pi_m$. Now, from (~\ref{X_Markov_relation}), as
$$
\mathcal{Q}_{j,i}(u,v)=\p(X_i^n=v\mid X_j^n=u)=\p(X_i^m=v\mid X_j^m=u),
$$
for any $n>m\geq j>i\geq 1$, assuming $\lim_{n\to \infty}\p(X_m^n=1)$ exists, for any $m\in \N$, we can conclude that $(X_m^n,\dots, X_1^n)$ has a weak limit, as $n\to \infty$. A natural framework to see this is through the reversed chain. From the time-reversal transformation, 
$$
\p(X_{i+1}^n=v\mid X_i^n=u)=\frac{\mathcal{Q}_{i+1,i}(v,u)\p(X_{i+1}^n=v)}{\p(X_i^n=u)}, \ u,v\in\{0,1\}.
$$

We need the following lemma.

\begin{lemma}\label{lemma-**}
Suppose
\begin{equation}\label{condition-**}
\sum_{j=1}^\infty p_j=\infty.
\end{equation}
Then $\varphi_i=\varphi_i(\mathbf{p}):=\lim\limits_{n\to \infty}\p(X_i^n=1)$ exists, and $0<\varphi_i<1$ for $i\geq 3$.
\end{lemma}

\begin{proof}
We have $\sum_{j=1}^\infty p_j=\infty$ if and only if $\prod_{i=j}^\infty q_i=0$ for any $j\in\N$. Now from Lemma~\ref{ProbEta}, $\lim_{n\to\infty}\p(X_i^n=1)$ exists and $0<q_ip_{i+1}<\varphi_i<q_i<1$.
\end{proof}

Note that we always have $\varphi_1=1-\varphi_2=1$. In the rest of this section, we assume condition (\ref{condition-**}) holds, hence $\varphi_i\in(0,1)$ is well-defined for $i\geq 3$. We are now ready to formally define $X^{\infty,\mathbf{p}}=(X^{\infty}_i)_{i\geq 1}$ by $X^{\infty}_1=1$ and 
\begin{eqnarray}
\p_{\mathbf{p}}(X^{\infty}_{i+1}=1\mid X^{\infty}_i=0)&=&1-\p_{\mathbf{p}}(X^{\infty}_{i+1}=0\mid X^{\infty}_i=0)\nonumber \\[6pt] 
&=&\lim_{n\to\infty}\frac{\p_{\mathbf{p}}(X^{n}_{i+1}=1)}{\p_{\mathbf{p}}(X^{n}_i=0)}=\frac{\varphi_{i+1}(\mathbf{p})}{1-\varphi_{i}(\mathbf{p})}, \label{transition-infinite}
\end{eqnarray}
and 
$$
\p_{\mathbf{p}}(X^{\infty}_{i+1}=1\mid X^{\infty}_i=1)=1-\p_{\mathbf{p}}(X^{\infty}_{i+1}=0\mid X^{\infty}_i=1)=0.
$$
Letting $X_i^{n,\mathbf{p}}=0$, for $i>n+1$, we have as $n\to\infty$
$$
(X^{n,\mathbf{p}}_1, X^{n,\mathbf{p}}_2,\dots)\Rightarrow (X^{\infty,\mathbf{p}}_1,X^{\infty,\mathbf{p}}_2,\dots).
$$
To see this notice that,  for any $m\in\N$ and $x=(x_i)_{i=1}^\infty\in \{0,1\}^{\N}$ with $x_1=1$ and $x_2=0$
\begin{eqnarray*}
\lim\limits_{n\to\infty}\p_{\mathbf{p}}(X^{n}_i=x_i;i=1,\dots,m)&=&\lim\limits_{n\to\infty}\p_{\mathbf{p}}(X^{n}_m=x_m)\prod_{i=1}^{m-1}\mathcal{Q}_{i+1,i}(x_{i+1},x_i)\\[6pt]
&=&\prod_{i=1}^{m-1}\mathcal{Q}_{i+1,i}(x_{i+1},x_i)\frac{\lim\limits_{n\to\infty}\p_{\mathbf{p}}(X^{n}_{i+1}=x_{i+1})}{\lim\limits_{n\to\infty}\p_{\mathbf{p}}(X^{n}_{i}=x_{i})}\\[6pt]
&=&\p_{\mathbf{p}}(X^{\infty}_i=x_i;i=1,\dots,m),
\end{eqnarray*}
where the last equality follows from the time-reversal transformation. Denote by $\mu_n$ the law of $(X_1^n,\cdots,X_n^n)$ and by $\mu$ the law of $X^{\infty,\mathbf{p}}$, and recall the definition of $\Delta_n$ from (\ref{Deltaj}). The next theorem gives the total variation distance between $\mu\pi_n^{-1}$ and $\mu_n$.

\begin{theorem}\label{TV-inf-ch}
For any $n\in\N$, 
$$
d_{TV}(\mu\pi^{-1}_n,\mu_n)=\varphi_n \mathbbm{1}\{n>1\}.
$$
\end{theorem}

\begin{proof}
For $n=1,2$, both sides of the above equality equal $0$. For $n>2$, write
\begin{eqnarray*}
\lefteqn{d_{TV}(\mu\pi^{-1}_n,\mu_n)}\\
&=&\hskip -.06in \frac{1}{2}\sum\limits_{x\in \Delta_n}\left|\p_p(X_i^n=x_i; \ i=1,\cdots,n)-\p_p(X_i^\infty=x_i; \ i=1,\cdots,n)\right|\\
&&+\frac{1}{2}\sum\limits_{x\in \Delta_{n-1}}\p_p(X_n^\infty=1, X_i^\infty=x_i; \ i=1,\cdots,n-1)\\
&=&\hskip -.06in \frac{1}{2}\sum_{x\in\Delta_n}\left|\left(1-\prod_{i=1}^{n-1}\frac{\lim\limits_{n\to\infty}\p_{\mathbf{p}}(X^{n}_{i+1}=x_{i+1})}{\lim\limits_{n\to\infty}\p_{\mathbf{p}}(X^{n}_{i}=x_{i})}\right)\prod_{i=1}^{n-1}\mathcal{Q}_{i+1,i}(x_{i+1},x_i)\right|\\[4pt]
&&\hskip -.2in +\frac{1}{2}\sum_{x\in\Delta_{n-1}}\frac{\mathcal{Q}_{n,n-1}(1,0)\varphi_n}{\lim\limits_{n\to \infty} \p_p(X_{n-1}=x_{n-1})}\prod_{i=1}^{n-2}\mathcal{Q}_{i+1,i}(x_{i+1},x_i)\frac{\lim\limits_{n\to\infty}\p_{\mathbf{p}}(X^{n}_{i+1}=x_{i+1})}{\lim\limits_{n\to\infty}\p_{\mathbf{p}}(X^{n}_{i}=x_{i})}\\[6pt]
&=&\hskip -.06in \frac{1}{2}\sum_{x\in\Delta_n}(1-\p_{\mathbf{p}}(X^{\infty}_n=0))\prod_{i=1}^{n-1}\mathcal{Q}_{i+1,i}(x_{i+1},x_i)\\
&&+\frac{1}{2}\sum_{x\in\Delta_{n-1}}\mathcal{Q}_{n,n-1}(1,0)\varphi_n \prod_{i=1}^{n-2}\mathcal{Q}_{i+1,i}(x_{i+1},x_i)=\frac{1}{2}\varphi_n+\frac{1}{2}\varphi_n=\varphi_n.
\end{eqnarray*}
\end{proof}

We notice that (\ref{condition-**}) does not guarantee $d_{TV}(\mu\pi^{-1}_n,\mu_n)=\varphi_n\to 0$, as $n\to\infty$. For instance, if $q_n=q\in (0,1)$ for $n\in\N$, then $\varphi$ and $X^{\infty,\mathbf{p}}$ are well-defined but $\varphi_n=q/(1+q)>0$.  Note that
$
q_n-q_n q_{n+1}\leq \varphi_n\leq q_n, 
$
and 
$
q_n(1-\varphi_{n+1})=\varphi_n.
$
Therefore, as $n\to\infty$, $q_n\to 0$ if and only if $\varphi_n\to 0$, hence the following result.

\newpage
\begin{corollary}\label{TV-converge}
The following are equivalent.
\begin{itemize}
\item[(i)] $\sum\limits_{n=1}^\infty p_n=\infty$ and as $n\to \infty$,
$
d_{TV}(\mu\pi^{-1}_n,\mu_n)\to 0.
$
\item[(ii)] As $n\to \infty$, $q_n\to 0$.
\end{itemize}
\end{corollary}

As examples, $\lim_{n\rightarrow\infty} q_n=0$, for both $\eta^{n,\theta},\tilde\eta^{n,\theta}$, and hence $\eta^{\infty,\theta},\tilde\eta^{\infty,\theta}$ are well-defined. Furthermore, $0<d_{TV}(\mu\pi_n^{-1},\mu_n)=\varphi_n\rightarrow 0$, as $n\rightarrow \infty$, for both processes. More exactly, from Lemma \ref{marginal-eta}, for $\eta$ we have
\begin{equation}\label{phi-eta}
\varphi_n^\eta=\lim_{m\to\infty}\p_\theta(\eta_n^m=1)=\p_\theta(\eta_n^\infty =1)=\theta\int_0^1 e^{-\theta u}(1-u)^{\theta+n-2} du.   
\end{equation}
For $\tilde\eta$, from the conditional relation between $\tilde\eta$ and $\tilde\xi$, given in~\cite{sjt2021} and this paper,
\[
\varphi_n^{\tilde\eta}=\lim\limits_{m\rightarrow \infty} \p_\theta(\tilde{\eta}_n^{m}=1)  = \p_\theta(\tilde{\eta}_n^{\infty}=1) =\frac{\lambda_{n+1,\infty}(\theta)\p_\theta(\tilde\xi_n=1)\lambda_{n-1}(\theta)}{\lambda_\infty(\theta)},
\]
where $\lambda_n(\theta)$, the probability that a $\theta$-biased random permutation of size $n$ is a derangement, given in Equation $(7)$ in~\cite{sjt2021},
\[
\lambda_{n+1,\infty}:=\p_\theta(\tilde\xi_{n+1}+\sum\limits_{j=n+1}^\infty\tilde\xi_j\tilde\xi_{j+1}=0)=M(\theta+1,\theta+n,-\theta),
\]
as given in Theorem 6 in~\cite{sjt2021}, and
\[
\lambda_\infty(\theta):=\p_\theta(\sum\limits_{j=1}^\infty\tilde\xi_j\tilde\xi_{j+1}=0)=\lim\limits_{n\to\infty}\lambda_n(\theta)=e^{-\theta}.
\]
Therefore,
\[
\varphi_n^{\tilde\eta}= \frac{\theta e^\theta \lambda_{n-1}(\theta)}{\theta+n-1}M(\theta+1,\theta+n,-\theta).
\]
The transition matrices of the limit Markov chains $\eta^\infty$ and $\tilde\eta^\infty$ can be easily obtained from (\ref{transition-infinite}).

\subsection{Conditional relation between $X^{\infty,\mathbf{p}}$ and $Y^{\bm\theta}$}
\label{sec:condRelationXinfinite}

Theorem~\ref{irvgeneral} implies that, when $\theta_i=(i-1)q_i/(p_i p_{i-1})$, for $i\geq 3$
$$
\p_{\mathbf{p}}(X^{n}_i=1)=1-\p_{\mathbf{p}}(X^{n}_i=0)=\frac{\theta_i\gamma_{i+1,n}(\bm{\theta})\gamma_{i-1}(\bm{\theta})/(i-1+\theta_i)}{\gamma_{n}(\bm{\theta})},
$$
where $\gamma_0(\bm\theta)=1$, $\gamma_{n,n}(\bm\theta)=\p_{\bm\theta}(Y_n=0)$, $\gamma_{n+1,n}(\bm\theta)=0$, and for $2\leq j<n$
$$
\gamma_{j,n}(\bm\theta)=\p_{\bm\theta}\left(Y_j+Y_n+\sum_{k=j}^{n-1}Y_kY_{k+1}=0\right).
$$
This reduces the transition probabilities of a finite reversed chain to
\begin{equation}\label{eqcond1}
\begin{split}
\p_{\mathbf{p}}(X_{i+1}^n=1 &\mid X_i^n=0)= 1-\p_{\mathbf{p}}(X_{i+1}^n=0\mid X_i^n=0)\\
=&\frac{\p_p(X_{i+1}^n=1)}{\p_p(X_i^n=0)}=\frac{\theta_{i+1}\gamma_{i+2,n}(\bm\theta)}{\theta_{i+1}\gamma_{i+2,n}(\bm\theta)+(i+\theta_{i+1})\gamma_{i+1,n}(\bm\theta)}.
\end{split}
\end{equation}
We look for conditions under which we can extend the conditional relation for $X^\infty$ and $Y$. We first record some useful properties of $Y$.  For $i\geq 2$, let
$$
\gamma_{i,\infty}(\bm\theta):=\p(Y_i+\sum_{j=i}^\infty Y_jY_{j+1}=0),
$$
and $
\gamma_{\infty}(\bm\theta):=\gamma_{2,\infty}(\bm\theta)$. Consider the following conditions
\begin{equation}\label{eqcond2}
\sum_{i=1}^\infty\frac{\theta_i\theta_{i+1}}{(i-1+\theta_i)(i+\theta_{i+1})}<\infty,
\end{equation}

\begin{equation}\label{eqcond3}
\theta_n/n\to 0, \ \ n\to \infty,
\end{equation}

\begin{equation}\label{eqcond4}
\sum_{i=1}^\infty \left(\frac{\theta_i}{i-1+\theta_i}\right)^2<\infty.
\end{equation}

Note that condition (\ref{eqcond3}) is equivalent to $\theta_n/(n-1+\theta_n)\to 0$,as $n\to \infty$. Also, (\ref{eqcond4}) implies (\ref{eqcond2}) and (\ref{eqcond3}), but not vice versa. For $j\geq 2$, let 
$$
\tilde{C}_j(n)=Y_{n-j+1}\prod_{l=2}^j(1-Y_{n-j+l})+\sum_{i=1}^{n-j}Y_iY_{i+j}\prod_{l=1}^{j-1}(1-Y_{i+l}),
$$
and $\tilde{C}_1(n)=Y_n+\sum_{i=1}^{n-1}Y_iY_{i+1}.$ Similarly, for $j\geq 2$,
$$
\tilde{C}_1(\infty)=\sum_{i=1}^{\infty}Y_iY_{i+1}, \ \ \text{and} \ \ \tilde{C}_j(\infty)= \sum_{i=1}^{\infty}Y_iY_{i+j}\prod_{l=1}^{j-1}(1-Y_{i+l}).
$$
From the definition $\gamma_\infty(\bm\theta)=\p_{\bm\theta}(\tilde C_1(\infty)=0)$. We have the following result.

\begin{theorem}\label{Y-limit-positive}
\begin{itemize}
\item[(i)] If (\ref{eqcond2}) and (\ref{eqcond3}) hold, then as $n\to \infty$, $\gamma_{n,\infty}(\bm\theta)\to 1$ and $\gamma_{i,n}(\bm\theta)\to \gamma_{i,\infty}(\bm\theta),$ for $i\geq 2$. In particular, $\gamma_n(\bm\theta) \to \gamma_\infty(\bm\theta)$, as $n\to \infty$.
\item[(ii)] If (\ref{eqcond2}) holds, then $\gamma_\infty(\bm\theta)>0$, $\gamma_{i,\infty}(\bm\theta)>0$, $i\in\N$.
\item[(iii)] If (\ref{eqcond4}) holds, then for any $m\in\N$,
\begin{equation}\label{dtvconv}
d_{TV}\left(\mathcal{L}(\tilde C_1(n),\dots,\tilde C_m(n)),\mathcal{L}(\tilde C_1(\infty),\dots,\tilde C_m(\infty))\right)\to 0,
\end{equation}
as $n\to \infty$. If in addition $\bm\theta$ is a bounded sequence, then (\ref{dtvconv}) holds for $m=m(n)=o(n)$, as $n\to \infty$.
\end{itemize}
\end{theorem}

\begin{proof}
For $(i)$, applying Borel-Cantelli lemma, we can write
\begin{equation}\label{limY.i.o.}
\lim_{n\to \infty}\p_{\bm\theta}\left(\sum_{j=n}^\infty Y_jY_{j+1}=0\right)=1-\p_{\bm\theta}(Y_n=Y_{n+1}=1 \ \ i.o.)=1;
\end{equation}
$$
\lim_{n\to \infty} \gamma_{n,\infty}(\bm\theta)=\lim_{n\to \infty}\p_{\bm\theta}(Y_n=0)\p_{\bm\theta}\left(\sum_{j=n+1}^\infty Y_jY_{j+1}=0\right)=1.
$$
Now, we have
\[
\begin{split}
\gamma_{i,\infty}&(\bm\theta)= \lim_{n\to\infty}\sum_{\ell=1}^2 \p_{\bm\theta}(Y_i+\sum_{j=i}^\infty Y_j Y_{j+1}=0,Y_n=\ell)\\
=&\lim_{n\to \infty} \left(\frac{\gamma_{i,n}(\bm\theta)\gamma_{n,\infty}(\bm\theta)}{\p_{\bm\theta}(Y_n=0)}+\gamma_{i,n-1}(\bm\theta)\p_{\bm\theta}(Y_n=1)\gamma_{n+1,\infty}(\bm\theta)\right)=\lim_{n\to \infty} \gamma_{i,n}(\bm\theta),
\end{split}
\]
since $\gamma_{n,\infty}(\bm\theta),\p_{\bm\theta}(Y_n=0)\to 1$ and $\p_{\bm\theta}(Y_n=1)\to 0$, as $n\to\infty$. Noting $\gamma_{2,n}(\bm\theta)=\gamma_n(\bm\theta)$ and $\gamma_{2,\infty}(\bm\theta)=\gamma_\infty(\bm\theta)$, finishes $(i)$.

For $(ii)$, note that from $(\ref{limY.i.o.})$, $\p_{\bm\theta}(\sum_{j=n}^\infty Y_jY_{j+1}=0)$ gets very close to $1$ for large $n$, and therefore 
\begin{eqnarray*}
\gamma_{i,\infty}(\bm\theta)\geq \p_{\bm\theta}\left(Y_n=0, \  Y_i+\sum_{j=i}^{\infty} Y_jY_{j+1}=0\right)= \gamma_{i,n}(\bm\theta)\p_{\bm\theta}\left(\sum_{j=n+1}^\infty Y_jY_{j+1}=0\right)>0.
\end{eqnarray*}
For $(iii)$, note that
\begin{eqnarray*}
\lefteqn{d_{TV}\left(\mathcal{L}(\tilde C_1(n),\dots,\tilde C_m(n)),\mathcal{L}(\tilde C_1(\infty),\dots,\tilde C_m(\infty))\right)}\\
&\leq & \p_{\bm\theta}\left((\tilde C_1(n),\dots,\tilde C_m(n))\neq (\tilde C_1(\infty),\dots,\tilde C_m(\infty))\right)\\
&\leq &  \p_{\bm\theta}\left(\{Y_{n-m+1}=\dots =Y_n=0\}^c\right) \p_{\bm\theta}(Y_{n+1}=0)\\
&& +\sum_{i\geq n}\p(Y_i=1)\p_{\bm\theta}\left(\{Y_{i+1}=\dots =Y_{i+m}=0\}^c\right)
\end{eqnarray*}
which is bounded by
\begin{eqnarray*}
&\leq & \frac{n}{n+\theta_{n+1}}\sum_{j=1}^m\frac{\theta_{n-m+j}}{n-m+j-1+\theta_{n-m+j}}+\sum_{\substack{i\geq n\\ 1\leq j\leq m}}\frac{\theta_i \theta_{i+j}}{(i-1+\theta_i)(i+j-1+\theta_{i+j})}\\
&\leq & \frac{n}{n+\theta_{n+1}}\sum_{j=1}^m\frac{\theta_{n-m+j}}{n-m+j-1+\theta_{n-m+j}}+m\sum_{i\geq n}\left(\frac{\theta_i}{i-1+\theta_i}\right)^2.
\end{eqnarray*}
For fixed $m\in \N$, the r.h.s. converges to $0$, as $n\to\infty$, if (\ref{eqcond4}) holds. Now if $\theta_n\leq c$ for $n\in \N$, the first and second terms on the right of the last inequality are bounded by $c\ m(n)/(n-m(n))$ and $m(n)\ c^2/(n-1)$. Therefore, the r.h.s. converges to $0$, if $m(n)=o(n)$.
\end{proof}
\begin{remark}\label{remark-C_1}
Note that $\{Y_{2i-1}Y_{2i}=1\}$ are independent for $i\in \N$, and also $\{Y_{2i}Y_{2i+1}=1\}$ are independent for $i\in \N$. Hence from Borel-Cantelli lemma, $\tilde C_1(\infty)<\infty$ a.s. if and only if $\gamma_{i,\infty}(\bm\theta)>0$ for $i\in \N$, if and only if (\ref{eqcond2}) holds.
\end{remark}
\begin{remark}
We can indeed show that for every $i\geq 2$, $\gamma_{i,n}(\bm\theta)\to\gamma_{i,\infty}(\bm\theta)>0$, as $n\to\infty$, if and only if (\ref{eqcond2}) and (\ref{eqcond3}) hold.

\begin{proof} The proof of sufficiency was given in Theorem~\ref{Y-limit-positive}, part $(ii)$. For the necessity part, consider $(p_i)_{i=3}^\infty$, $p_1=0, p_2=1$ s.t. $\theta_i=(i-1)q_i/(p_ip_{i-1})$, for $i\geq 3$. From (\ref{gammai}) and (\ref{pq}), for $n\geq 4$, we can write
\begin{equation}\label{gammmai2}
\gamma_n(\bm\theta)=\frac{\p_{\bm\theta}(Y_n=0)}{1+\theta_2}\prod_{j=2}^{n-2}\frac{p_j}{p_jp_{j+1}+q_{j+1}},
\end{equation}
where the limit of the product on the right of the last equation always exists, as $p_j/(p_j p_{j+1}+q_{j+1})\leq 1$, for $j\geq 2$. Furthermore, this limit is strictly positive if and only if
\begin{equation}\label{eqcond2-weaker}
\sum\limits_{j=3}^\infty\left(1-\frac{p_j}{p_j p_{j+1}+q_{j+1}}\right)=\sum_{j=3}^\infty \frac{q_jq_{j+1}}{p_j p_{j+1}+q_{j+1}}<\infty.
\end{equation}
On the other hand, as $p_{j-1} p_{j}+q_{j}\leq 1$,
\[
\begin{split}
\sum_{j=3}^\infty \frac{q_j q_{j+1}}{p_j p_{j+1}+q_{j+1}} \leq & \sum_{j=3}^\infty \frac{q_j q_{j+1}}{(p_{j-1} p_j+q_j)(p_j p_{j+1}+q_{j+1})}= \sum_{j=3}^\infty\frac{\theta_j\theta_{j+1}}{(j-1+\theta_i)(j+\theta_{j+1})},
\end{split}
\]
where the last equality follows by substituting $\theta_i=(i-1)q_i/(p_i p_{i-1})$, for $i\geq 3$. Thus (\ref{eqcond2}) implies (\ref{eqcond2-weaker}).

Now to complete the proof of the necessity part, note that from Theorem~\ref{Y-limit-positive}, part $(ii)$ and Remark~\ref{remark-C_1}, $\gamma_{i,\infty}(\bm\theta)=0$ if (\ref{eqcond2}) does not hold. On the other hand, if (\ref{eqcond2}) holds but $\lim_{n\to\infty}\p_{\bm\theta}(Y_n=0)$ does not exist, the limit of the product on the r.h.s. of (\ref{gammmai2}) is strictly positive, as (\ref{eqcond2-weaker}) holds. Therefore, from (\ref{gammmai2}), $\lim_{n\to\infty}\gamma_n(\bm\theta)$ does not exist. This shows that, if $\gamma_n\to \gamma_\infty>0$, as $n\to\infty$, then (\ref{eqcond2}) and (\ref{eqcond3}) both hold.

For the general case of $\gamma_{i,n}\to\gamma_{i,\infty}>0$, as $n\to\infty$, use the same argument for the shifted sequence $\bar{Y}_j=Y_{j+i-2}$, $j\geq 2$, with
\[
\p_{\bm\theta}(\bar{Y}_j=0)=\frac{\bar\theta_j}{j-1+\bar{\theta}_j}=\frac{\theta_{j+i-2}}{\theta_{j+i-2}+j+i-3},
\]
which is equivalent to take $\bar{\theta}_j=(j-1)\theta_{j+i-2}/(j+i-3)$, for $j\geq 2$. This finishes the proof of the remark.
\end{proof}
\end{remark}

The following theorem provides conditions under which $X^{\infty,\mathbf{p}}$ is well-defined, while its law coincides with the law of $Y^{\bm\theta}$ conditional on no fixed point.

\begin{theorem}\label{Thm-cond-law-infty}
Suppose that
\begin{enumerate}
\item[(i)] $\theta_n=(n-1)q_n/(p_np_{n-1})$, for $n\geq 3$;
\item[(ii)] as $n\to \infty$, $q_n\to 0$ (or equivalently $\theta_n/n\to 0$);
\item[(iii)] $\sum_{i=3}^\infty\frac{q_iq_{i+1}}{(q_i+p_ip_{i-1})(q_{i+1}+p_{i+1}p_i)}= \sum_{i=3}^\infty\frac{\theta_i\theta_{i+1}}{(i-1+\theta_i)(i+\theta_{i+1})}<\infty$.
\end{enumerate}
Then
\begin{equation}\label{condLawXY}
\mathcal{L}(X^{\infty,\mathbf{p}})=\mathcal{L}(Y^{\bm\theta}\mid\tilde C_1(\infty)=0).
\end{equation}

Moreover, under assumptions $(i)$, $(ii)$ and $(iii)$,
\begin{eqnarray*}
\p_{\mathbf{p}}(X^\infty_{i+1}=1\mid X_i^\infty=0)&=&1-\p_{\mathbf{p}}(X^\infty_{i+1}=0\mid X_i^\infty=0)\\[6pt]
&=&\frac{\theta_{i+1}\gamma_{i+2,\infty}(\bm\theta)}{\theta_{i+1}\gamma_{i+2,\infty}(\bm\theta)+(i+\theta_{i+1})\gamma_{i+1,\infty}(\bm\theta)},
\end{eqnarray*}
and 
$$
\p_{\mathbf{p}}(X^\infty_{i+1}=1\mid X_i^\infty=1)=1-\p_{\mathbf{p}}(X^\infty_{i+1}=0\mid X_i^\infty=1)=0.
$$
\end{theorem}

\begin{proof}
We first notice that, assuming $(i)$
$$
\frac{(n-1)q_n}{n}\leq \frac{\theta_n}{n}= \frac{(n-1)q_n}{n p_np_{n-1}},
$$
hence $q_n\to 0$ if and only if $\theta_n/n\to 0$, as $n\to\infty$. Also the equality in $(iii)$ is a direct result of applying $(i)$. Now, from Theorem~\ref{Y-limit-positive}, $\gamma_\infty(\bm\theta),\gamma_{i,\infty}(\bm\theta)>0$, $i\geq 2$, and thus from Corollary~\ref{TV-converge} and Theorem~\ref{irvgeneral},
\begin{eqnarray*}
\p_{\mathbf{p}}(X^\infty_{i+1}=1\mid X_i^\infty=0)&=&\frac{\varphi_{i+1}(\mathbf{p})}{1-\varphi_i(\mathbf{p})}=\lim_{n\to \infty}\frac{\p_{\mathbf{p}}(X^n_{i+1}=1)}{\p_{\mathbf{p}}(X^n_{i}=0)}\\[4pt]
&=&\lim_{n\to \infty}\frac{\theta_{i+1}\gamma_{i+2,n}(\bm\theta)}{\theta_{i+1}\gamma_{i+2,n}(\bm\theta)+(i+\theta_{i+1})\gamma_{i+1,n}(\bm\theta)} \ \ (from~(\ref{eqcond1}))\\[4pt]
&=&\frac{\theta_{i+1}\gamma_{i+2,\infty}(\bm\theta)}{\theta_{i+1}\gamma_{i+2,\infty}(\bm\theta)+(i+\theta_{i+1})\gamma_{i+1,\infty}(\bm\theta)} \ \ (from \ \ Theorem~\ref{Y-limit-positive})\\[4pt]
&=&\frac{\frac{\theta_{i+1}}{i+\theta_{i+1}}\gamma_{i+2,\infty}(\bm\theta)}{\frac{\theta_{i+1}}{i+\theta_{i+1}}\gamma_{i+2,\infty}(\bm\theta)+\gamma_{i+1,\infty}(\bm\theta)}\cdot \frac{\gamma_i(\bm\theta)/\gamma_\infty(\bm\theta)}{\gamma_i(\bm\theta)/\gamma_\infty(\bm\theta)}\\[4pt]
&=&\frac{ \p_{\bm{\theta}}(Y_{i+1}=1,Y_i=0\mid \tilde C_1(\infty)=0)}{\p_{\bm{\theta}}(Y_i=0\mid \tilde{C}_1(\infty)=0)},
\end{eqnarray*}
for any $i\geq 2$, hence (\ref{condLawXY}) holds.
\end{proof}

\begin{remark}
Condition $(ii)$ ensures that $X^{\infty,\mathbf{p}}$ is well-defined and serves as the weak limit of $X^{n,\mathbf{p}}$ as $n\to \infty$. Assuming $(iii)$, we get $\gamma_\infty(\bm\theta), \gamma_{i,\infty}(\bm\theta)>0$, and the conditional relation (\ref{condLawXY}). Note that $(iii)$ in Theorem~\ref{Thm-cond-law-infty} holds for $\eta$ and $\tilde{\eta}$, so does the conditional relation (\ref{condLawXY}) for both.
\end{remark}

In the last section, the transition probabilities of $\eta^{\infty,\theta}$ were given in terms of $\varphi_n^\eta$. But it is still useful to find the exact value of $\delta_{i,\infty}$.

\begin{proposition}
For any $i\geq 4$ and $\theta>0$,
\begin{multline*}
   \delta_{i,\infty}(\theta):=\p_{\theta}(\xi_i+\sum_{j=i}^\infty\xi_j\xi_{j+1}=0)=M(1,\theta+i-1,-\theta) \ \frac{B(z_1(\theta)+i-3,z_2(\theta)+i-3)}{B(i-2,\theta+i-1)}. 
\end{multline*}
where $z_1(\theta),z_2(\theta)$ are defined as in Corollary~\ref{PGC}.
\end{proposition}

\begin{proof}
From Theorem~\ref{Y-limit-positive}, $\lim_{n\to\infty}\delta_{i,n}(\theta)=\delta_{i,\infty}(\theta)>0$, and in particular, from Corollary~\ref{PGC}, getting the limit of (\ref{deltai}), we have 
\begin{equation}\label{deltainft}
0<\delta_\infty(\theta)=\lim_{n\to\infty}\delta_n(\theta)=\frac{(\theta^2+\theta+2)\Gamma(z_1(\theta))\Gamma(z_2(\theta))}{(1+\theta_2^*)\Gamma(z_1(\theta)+z_2(\theta))}.
\end{equation}

Now, for $i\geq 4$,
$$
\varphi_{i-1}=\p_\theta(\eta_{i-1}^\infty=1)=\frac{\delta_{i-2}(\theta)\frac{\theta^*_{i-1}}{i-2+\theta^*_{i-1}}\delta_{i,\infty}(\theta)}{\delta_\infty(\theta)},
$$
recalling $\theta^*_i=\theta(1+\theta/(i-2))$, for $i\geq 4$, and $\theta^*_3=\theta$. Hence $\delta_{4,\infty}=(\theta+2)\varphi_3\delta_\infty/(\theta\delta_2)$ and
$$
\delta_{i,\infty}(\theta)=\frac{\varphi_{i-1}((i-2)(i-3)+\theta(i-3+\theta))\delta_\infty(\theta)}{\theta(i-3+\theta)\delta_{i-2}(\theta)}, \ \ i\geq 5.
$$
Applying (\ref{deltai}), (\ref{phi-eta}), (\ref{deltainft}), after simplification, this reduces to
\[
\frac{\theta_{(3)}B(z_1(\theta),z_2(\theta))(z_1(\theta))_{(i-3)}(z_2(\theta))_{(i-3)}}{(i-3)!\theta_{(i-2)}}\int_0^1e^{-\theta u}(1-u)^{\theta+i-3}du,
\]
which concludes the proposition, after further simplification.
\end{proof}

\begin{remark}
Note that $\delta_{2,\infty}(\theta)=\delta_\infty(\theta)$ given in (\ref{deltainft}), and $\delta_{3,\infty}(\theta)$ can be obtained using the recursion
$$
\gamma_{i,\infty}(\bm\theta)=\frac{i-1}{i-1+\theta_i}\left(\gamma_{i+1,\infty}(\bm\theta)+\frac{\theta_{i+1}}{i+\theta_{i+1}}\gamma_{i+2,\infty}(\bm\theta)\right),  \ \ i\geq 3.
$$
\end{remark}

\section{A coupling between $X$ and $Y$: a push-forward relation}\label{sec:coupling2}
We have seen so far that, for $\theta>0$, the conditional relation (\ref{condlaw}) does not hold between $\eta^{n,\theta}$ and the classic Feller coupling $(\tilde\xi_n^\theta,\cdots,\tilde\xi_1^\theta)$, given there is no $11$ patterns observed in the latter. In this section, we will see that in fact the distributions of these two are related by a simple push-forward relation, in the sense that the law of $\eta^{n,\theta}$ is the image of that of $(\tilde\xi_n^\theta,\cdots,\tilde\xi_1^\theta)$ under a natural $11$-erasing mapping. More generally, in this section, we establish a push-forward relation between $X^{n,\mathbf{p}}$ and $Y^{n,\bm\theta}$, for
\begin{equation}\label{couple}
p_i=\frac{i-1}{i-1+\theta_i},
\end{equation}
or equivalently $\theta_i=(i-1)q_i/p_i$, for $n\in \N$ and $3\leq i\leq n-1$. Under (\ref{condition-**}), which ensures the existence of the Markov chain $X^{\infty,\mathbf{p}}$, we also provide a similar coupling relation between $X^{\infty,\mathbf{p}}$ and $Y^{\bm\theta}$, where once again $p_i$ and $\theta_i$ are related by (\ref{couple}), for $i\geq 3$. We use the coupling relations to prove a central limit theorem for $K_n$.  We also see that when  $\theta_n \rightarrow\theta>0$, as $n\rightarrow \infty$, the asymptotic behavior of the joint distribution of the normalized cycle lengths of $X^{n,\textbf{p}}$, in order of their formation, can be studied through that of $Y^{n,\bm\theta}$. 

To make this precise, suppose (\ref{couple}) holds for $\theta_i$ and $p_i$, for $3\leq i\in \N$. For any $n\in \N$, and $\mathbf{y}=(y_1,y_2,\cdots)\in \{1\}\times \{0,1\}^{\N}$, let
\[
\begin{split}
\beta_i^n(\mathbf{y}):=& \mathbbm{1}\{y_i=1\} \cdot \max\{0\leq j\leq n-i-1 \ : \ \prod_{k=0}^j y_{i+k}=1\};\\
\beta_i^\infty(\mathbf{y}):=& \mathbbm{1}\{y_i=1\} \cdot \sup\{0\leq j \ : \ \prod_{k=0}^j y_{i+k}=1\}.\\
\end{split}
\]
In fact, $\beta_i^n(\mathbf y)$ and $\beta_i^\infty(\mathbf y)$ count the number of consecutive $1$'s right after $y_i=1$, in $(y_{i+1},\cdots,y_{n-1})$ and $(y_{i+1},y_{i+2},\cdots)$, respectively. Note that $\beta_i^\infty(\mathbf y)=\infty$ for some $i\in \N$, if and only if there exists a unique $r\leq i$ s.t. $(1-y_r)\prod_{j=1}^\infty y_{r+j}=1$. We now define the mappings $\chi_n$ and $\chi_\infty$ that erase the $11$ patterns in $Y^n$ and $Y$, starting from the end of the sequence. More explicitly, for $\mathbf{y}\in \{1\}\times\{0,1\}^\N$, and $n\in \N$, let $\Phi_1^n(\mathbf{y})=\Phi_1^\infty(\mathbf{y})=1$, $\Phi_2^n(\mathbf{y})=\Phi_2^\infty(\mathbf{y})=0$, $\Phi_j^n(\mathbf{y})=0$ for $j\geq n$, and let
\[
\begin{split}
&\Phi_i^n:=\mathbbm 1\{y_i=1, \beta_i^n(\mathbf{y}) \ \ \text{even}\}, \ \ 3\leq i\leq n-1;\\
&\Phi_i^\infty:=\mathbbm 1\{y_i=1, \beta_i^\infty\ \ \text{finite and even}\}, \ \ i\geq 3.
\end{split}
\]
It is clear from the definition that $\Phi^n$ and $\Phi^\infty$ indeed remove consecutive $1$'s (i.e. $11$ patterns) from $\mathbf y$, starting from the end. To make use of these for defining our mappings, for any $n\in \N$, let
\[
\overline\Delta_n:=\left\lbrace \mathbf y\in \{1\}\times \{0,1\}^{\N}, \ y_1=1, \ y_2+\sum_{j=2}^{n-1}y_j y_{j+1} +\sum_{j=n}^\infty y_j=0 \right\rbrace,
\]
\[
\Delta_\infty:=\left\lbrace \mathbf y\in \{1\}\times \{0,1\}^{\N}, \ y_1=1, \ y_2+\sum_{j=2}^{\infty}y_j y_{j+1}=0 \right\rbrace,
\]
and define $\chi_n: \{1\}\times \{0,1\}^{\N}\to \bar\Delta_n$ and $\chi_\infty: \{1\}\times \{0,1\}^{\N}\to \Delta_\infty$ by $\chi_n(\mathbf y)=(\Phi_i^n(\mathbf y))_{i=1}^\infty$ and $\chi_\infty(\mathbf y)=(\Phi_i^\infty(\mathbf y))_{i=1}^\infty$. For example, for ~$\mathbf{y}=(11111001111000...)$, we have ~$\chi_{11}(\mathbf y)=$  ~$(10101001010000...)$ and $\chi_{12}(\mathbf y)=(10101000101000...)$.

Let $F_i^{\bm\theta}(n)=\Phi_i^n(Y^{\bm\theta})$ and  $F_i^{\bm\theta}(\infty)=\Phi_i^\infty(Y^{\bm\theta})$ be the image of $Y^{\bm\theta}$ under these mappings. From the definition,
\[
\p(F_i^{\bm\theta}(n)=0 \mid F_{i+1}^{\bm\theta}(n)=0)=\p_{\bm\theta}(Y_i=0\mid \Phi_{i+1}^n(Y)=0)=\p_{\bm\theta}(Y_i=0) =p_i,
\]
\[
\p(F_i^{\bm\theta}(n)=0 \mid F_{i+1}^{\bm\theta}(n)=1)=\p_{\bm\theta}(Y_i\in \{0,1\}\mid \Phi_{i+1}^n(Y)=1)=1,
\]
which implies $\mathcal{L}(F_1^{\bm\theta}(n),F_2^{\bm\theta}(n),\cdots)=\mathcal{L}(X_1^{n,\mathbf{p}},X_2^{n,\mathbf{p}},\cdots)$, where $X_j^{n,\mathbf{p}}=0$, for $j\geq n$. Hence, for any $n\in \N$ and $\bm\theta$ and $\mathbf{p}$ related by (\ref{couple}),
$$
\mathcal{L}(X_1^{n,\mathbf{p}},X_2^{n,\mathbf{p}},\cdots)=\chi_n\ast\mathcal{L}(Y^{\bm\theta}),
$$
where the r.h.s. is the push-forward of the law of $Y^{\bm\theta}$ under $\chi_n$. Likewise, from the definition $\mathcal{L}(F_1^{\bm\theta}(\infty),F_2^{\bm\theta}(\infty),\cdots)=\chi_\infty\ast\mathcal{L}(Y^{\bm\theta})$. It follows from Borel-Cantelli lemma that $\p(\ \beta_i^\infty(Y^{\bm\theta})=\infty \ \text{for some} \ i\in \N)=0$ under (\ref{condition-**}). So for any $m\in \N$, $(F_1^{\bm\theta}(n),\cdots,F_m^{\bm\theta}(n))\to (F_1^{\bm\theta}(\infty),\cdots,F_m^{\bm\theta}(\infty))$ a.s., as $n\to\infty$, concluding $\mathcal{L}(X_1^{n,\mathbf{p}},X_2^{n,\mathbf{p}},\cdots)\Rightarrow \chi_\infty \ast \mathcal{L}(Y^{\bm\theta}).$ Therefore, we have the following representation of the limit chain
\[
\mathcal{L}(X^{\infty,\mathbf{p}})=\chi_\infty\ast\mathcal{L}(Y^{\bm\theta}).
\]

We can also see that, assuming (\ref{couple}), for any $i\geq 3$, we readily have
\[
\begin{split}
d_{TV}(\mathcal{L}(X^{n,\mathbf{p}}), \mathcal{L}(Y^{n,\bm\theta}))\leq & \p((F_1(n),\cdots, F_n(n))\neq (Y_1,\cdots,Y_n))\\
\leq & \p(Y_n^{\bm\theta}=1)+1-\gamma_{n}(\bm\theta),\\
\end{split}
\]
which leads to
\[
d_{TV}(\mathcal{L}(X^{\infty,\mathbf{p}}), \mathcal{L}(Y^{\bm\theta}))\leq 1-\gamma_{\infty}(\bm\theta),
\]
if we additionally assume (\ref{eqcond2}) and (\ref{eqcond3}).
In the rest of this paper, we assume that (\ref{couple}) holds and $X^{n,\mathbf{p}}$ and $Y^{n,\bm\theta}$ are coupled as explained above, i.e.  we assume $X^{n,\mathbf{p}}=(F_n(n),\cdots, F_1(n))$. 

\subsection{Central limit theorem for the number of cycles}\label{sec:clt}
The next theorem provides a central limit theorem for $X^{n,\mathbf{p}}$.
\begin{theorem}\label{thm:clt}
let $\bar q_n=\sum_{i=1}^n q_i$ and $\bar{\bar{q}}_n=\sum_{i=1}^n q_i^2$, and suppose $(\bar{\bar{q}}_n)^2/\bar q_n\to 0$ as $n\to \infty$. Then as $n\to \infty$,
$$
\frac{K_n^{\mathbf{p}}-\bar q_n}{\sqrt{\bar q_n}}\Rightarrow \mathcal{N}
$$
where $\mathcal{N}\sim Normal(0,1)$.
\end{theorem}

\begin{proof}
To prove the central limit theorem for $K_n^{\mathbf{p}}$, we consider $\tilde K_n^{\bm \theta}$ coupled with $K_n^{\mathbf{p}}$, where $\mathbf{p}$ and $\bm\theta$ are related by (\ref{couple}). Form the way we coupled $K_n$ and $\tilde K_n$, we conclude $\tilde K_n^{\bm \theta}-K_n^{\mathbf{p}}\leq (\tilde C_1(n)+1)/2$. From~\cite{bh84}, Theorem 1, we get 
\begin{equation*}
d_{TV}(\mcL(\tilde K_n), Po(\bar q_n))\leq \bar{\bar{q}}_n(1-e^{-\bar q_n})/\bar{q}_n.
\end{equation*}
Hence, from the assumption $\bar{\bar{q}}_n/\bar{q}_n\to 0$, as $n\to \infty$, so $d_{TV}(\mcL(\tilde K_n), Po(\bar q_n))\to 0$ as $n\to \infty$. Therefore, as $n\to \infty$
$$
\frac{\tilde K_n-\bar q_n}{\sqrt{\bar q_n}}\Rightarrow \mathcal{N}.
$$

It now suffices to prove $(\tilde C_1(n)-Y_n)/\sqrt{\bar q_n}\to 0$, in probability, as $n\to \infty$.This implies $(\tilde K_n-K_n)/\sqrt{\bar q_n}\to 0$, in probability, as $n\to \infty$, and hence
$$
\frac{K_n-\bar q_n}{\sqrt{\bar q_n}}=\frac{\tilde K_n-\bar q_n}{\sqrt{\bar q_n}}-\frac{\tilde K_n-K_n}{\sqrt{\bar q_n}}\Rightarrow \mathcal{N}.
$$

To complete the proof,  write $\tilde C_1(n)-Y_n=\sum_{i=1}^{n-1} Y_iY_{i+1}$,
$$
\e_{\bm\theta}[(\tilde C_1(n)-Y_n)^2]=\sum_{i=1}^{n-1}q_iq_{i+1}+2\sum_{i=1}^{n-2}q_iq_{i+1}q_{i+2}+2\sum_{i,j} q_iq_{i+1}q_jq_{j+1},
$$
where the last term of the r.h.s. is over all $i,j$ such that $2\leq i+1<j\leq n-1$. Now as $q_i^2+q_{i+1}^2\geq 2 q_iq_{i+1}$, for $i\in \N$, we get
$$
\sum_{i=1}^{n-2}q_iq_{i+1}q_{i+2}\leq \sum_{i=1}^{n-1}q_iq_{i+1}\leq \bar{\bar{q}}_n,
$$

$$
\sum_{i,j} q_iq_{i+1}q_jq_{j+1}\leq \left(\sum_{i=1}^{n-1}q_iq_{i+1}\right)^2\leq (\bar{\bar{q}}_n)^2.
$$

Hence, for any $\varepsilon>0$, as $n\to \infty$,
$$
\p\left(\tilde C_1(n)-Y_n\geq \varepsilon\sqrt{\bar q_n}\right)\leq \frac{\e_{\bm\theta}[(\tilde C_1(n)-Y_n)^2]}{\varepsilon^2 \bar{q}_n}\leq \frac{3\bar{\bar{q}}_n+(\bar{\bar{q}}_n)^2}{\varepsilon^2\bar q_n}\to 0.
$$
\end{proof}

Recall the definition of $\psi(\mathbf{p})$ from (\ref{psifunction}). The following is an immediate application of the last theorem.

\begin{corollary}\label{cor:clt}
Suppose $nq_n\to \theta>0$. Then $(K_n-\bar{q}_n)/\sqrt{\bar q_n}\Rightarrow \mathcal{N}$, as $n\to\infty$. Furthermore if $|\psi_\theta(\mathbf{p})|<\infty$, then as $n\to \infty$
\begin{equation}\label{clttheta}
\frac{K_n-\theta\log n}{\sqrt{\theta\log n}}\Rightarrow \mathcal{N}.
\end{equation}
 In particular, (\ref{clttheta}) holds if there exists $c>0$ such that $|\theta_n-\theta|\asymp n^{-c}$.
\end{corollary}

\subsection{ Ordered and longest cycles}\label{sec:orderedCycles}

Let $\tilde\A_j(n)$ be the length of the $j^{th}$ cycle in the natural order of the cycles determined by the GFC, with $\tilde\A_j(n)=0$ for $j> \tilde{K}_n$.  \ignore{Note that the probability that a $\bm \theta$-biased random permutation sampled from the GFC process has $\tilde K_n>r$ and $(\tilde \A_1(n),\cdots,\tilde{\A}_r(n))=(a_1,\cdots,a_r)$, with $a_1+\cdots+a_r=m<n$, is given by
$$
\frac{(\theta_{n+1-a_1}\theta_{n+1-a_1-a_2}\cdots\theta_{n+1-m})\bm\theta_{<n-m>}}{\bm\theta_{<n>}},
$$
and the number of such permutations is given by
\begin{multline*}
{{n-1}\choose{a_1-1}}(a_1-1)! {{n-1-a_1}\choose{a_2-1}}(a_2-1)! \cdots {{n-1-a_1\cdots-a_{r-1}}\choose{a_r-1}} (a_r-1)!\\\\
=\frac{n!}{n(n-a_1)\cdots(n-a_1-\cdots-a_{r-1})(n-m)!}.
\end{multline*}
Therefore,} It is easy to see that, for $r,a_1,\dots,a_r\in \N$, with $a_1+\cdots+a_r=m<n$, we have
\begin{multline}\label{ordercyclesY}
\p(\tilde\A_1(n)=a_1,\cdots ,\tilde\A_r(n)=a_r,\tilde K_n>r) \\= \frac{n!\bm\theta_{<n-m>}}{(n-m)!\bm\theta_{<n>}}\frac{\theta_{n+1-a_1}\theta_{n+1-a_1-a_2}\cdots\theta_{n+1-m}}{n(n-a_1)\cdots(n-a_1-\cdots-a_{r-1})}.
\end{multline}
Assuming $\theta_{n}\rightarrow \theta$ as $n\rightarrow \infty$, we have $\frac{\bm\theta_{<n-m>}}{\bm\theta_{<n>}}\sim \frac{\theta_{(n-m)}}{\theta_{(n)}}$, and therefore  it is straightforward from the corresponding result for the Feller Coupling that
\begin{eqnarray}
\lefteqn{
n^r \mathbb{P}(\tilde\A_1(n) =\lfloor nx_1\rfloor ,\tilde\A_2(n) = \lfloor nx_2\rfloor,\ldots,\tilde\A_r(n) = \lfloor nx_r\rfloor, K_n > r)} \nonumber\\
& \to & \frac{\theta^r (1 - x_1 - \cdots - x_r)^{\theta - 1}}{(1-x_1)(1 - x_1-x_2)\cdots(1-x_1-x_2 - \cdots - x_{r-1})} \label{limit}
\end{eqnarray}
for a fixed $r$, and $x_1,x_2,\cdots, x_r>0$ satisfying $x_1+\cdots+x_r<1$, which for $\theta_n\to\theta$ as $n\to\infty$, implies 
\begin{equation}\label{limit-A}
n^{-1}(\tilde\A_1(n),\tilde\A_2(n),\cdots)\Rightarrow(\A_1,\A_2,\cdots) \mbox{ as } n \to \infty,
\end{equation}
where $\A_1,\A_2,\cdots$ is GEM($\theta$), with density given in~\cite{abt03}, equation (5.28). 

Now, let $\A_j(n)$ be the length of the $j^{th}$ cycle in the natural order of the cycles determined by $X^{n,\mathbf{p}}$, with $\A_j(n)=0$ for $j>K_n$.  As in the GFC case,  we can deduce the asymptotic behavior of $n^{-1}(\A_1(n),\A_2(n),\ldots) $,  as stated in the following theorem.

\begin{theorem}\label{GEMlaw} Suppose $\theta_n\to\theta$ as $n\to \infty$. Then
$$
n^{-1}(\A_1(n),\A_2(n),\ldots) \Rightarrow (\A_1,\A_2,\ldots) \mbox{ as } n \to \infty.
$$
\end{theorem}
\begin{proof} For a fixed $r \in \N$ as $n\to\infty$, we show
\begin{eqnarray*}
\lefteqn{
n^r \mathbb{P}(\A_1(n) =m_1 ,\A_2(n) = m_2,\A_r(n) = m_r, K_n > r)}\\
& \to & \frac{\theta^r (1 - x_1 - \cdots - x_r)^{\theta - 1}}{(1-x_1)(1 - x_1-x_2)\cdots(1-x_1-x_2 - \cdots - x_{r-1})}
\end{eqnarray*}
for $a_i = \lfloor n x_i \rfloor, i = 1, \ldots,r$ satisfying $x_1,x_2,\ldots,x_r > 0, x_1 + \cdots + x_r < 1.$

We notice that the l.h.s. differs from the corresponding probability under the generalized Feller Coupling (l.h.s. in (\ref{limit})) by the reciprocal of 
\begin{multline*}
 \mathbb{P}(Y_n = 0,Y_{n-a_1} = 0, \ldots,Y_{n - a_1 - a_2 -\cdots - a_{r-1}} = 0, Y_2=0)\\
=
\frac{n-1}{(1+\theta_2)(n-1+\theta_{n})} \prod_{j=1}^{r-1} \frac{n-1-a_1- \cdots -a_{j}}{n-1-a_1 - \cdots - a_{j} + \theta_{n-a_1- \cdots -a_{j}}}\to 1   
\end{multline*}
 as $n \to \infty$. The result immediately follows from the corresponding result for the Feller Coupling.
\end{proof}

\ignore{
\begin{remark}
It follows from Theorem 3 of~\cite{dj89} that the ordered circle lengths $L_1(n) \geq L_2(n) \geq \cdots$
satisfy $$n^{-1}(L_1(n),L_2(n),\ldots) \Rightarrow (L_1,L_2,\ldots),$$ 
where $(L_1,L_2,\ldots)$ has the Poisson-Dirichlet law with parameter $\theta$. 
\end{remark}}

\medskip 
\noindent\textbf{Acknowledgments}
We thank Dr Kathy Ewens for bringing the children's playground game to our attention, and Dr. Will Stephenson for describing  a related game to us.

\bibliography{dSJT2022}
\bibliographystyle{plain}

\end{document}